\documentclass[11pt]{article}
\usepackage{amssymb, latexsym, mathtools, array, multirow, graphicx, indentfirst, enumerate}
\usepackage[all]{xy}

\usepackage[thmmarks, amsmath]{ntheorem}
{
\theoremstyle{nonumberplain}
\theoremheaderfont{\itshape}
\theorembodyfont{\upshape}
\theoremsymbol{\mbox{$\Box$}}
\newtheorem{proof}{Proof.}
}

\usepackage[hang,flushmargin]{footmisc}

\def\Z{\mathbb Z}
\def\Q{\mathbb Q}
\def\R{\mathbb R}
\def\C{\mathbb C}
\def\la{\lambda}

\theoremstyle{plain}

\newtheorem{lemma}{Lemma}[section]
\newtheorem{proposition}{Proposition}[section]
\newtheorem{theorem}{Theorem}[section]

\theorembodyfont{\upshape}
\newtheorem*{construction}{Canonical Construction}
\newtheorem*{observation}{Key Observation}
\newtheorem{definition}{Definition}[section]
\newtheorem{notation}{Notation}[section]
\newtheorem{example}{Example}[section]
\newtheorem{remark}{Remark}[section]

\theoremstyle{empty}

\newtheorem{thm}{Theorem}

\newcommand\blfootnote[1]{
\begingroup 
\renewcommand\thefootnote{}\footnote{#1}
\addtocounter{footnote}{-1}
\endgroup 
}

\title{On string quasitoric manifolds and their orbit polytopes}
\author{Qifan Shen}
\date{}
\begin{document}

\renewcommand{\theequation}{\thesection.\arabic{equation}}
\setcounter{equation}{0} \maketitle

\vspace{-1.0cm}

\bigskip

\noindent{\bf Abstract.} This article mainly aims to give combinatorial characterizations and topological descriptions of quasitoric manifolds with string property. We provide a necessary and sufficient condition for a simple polytope in dimension 2 and 3 to be realizable as the orbit polytope of a string quasitoric manifold. In particular, a complete description of string quasitoric manifolds over prisms is obtained. On the other hand, we characterize string quasitoric manifolds over $n$-dimensional simple polytopes with no more than $2n+2$ facets. Further results are available when the orbit polytope is the connected sum of a cube and another simple polytope. In addtion, a real analogue concerning small cover is briefly discussed. 

\blfootnote{
Partially supported by the grant from NSFC (No. 11971112).}

\medskip

\renewcommand{\theequation}{\thesection.\arabic{equation}}
\section{Introduction} \label{introduction}
\setcounter{equation}{0}

String structure is a higher version of spin structure. The concept originates from quantum field theory in physics, in which vanishing of a certain \emph{quantum anomaly} (the failure of a symmetry to be preserved at a quantum level) is guaranteed by existence of its corresponding geometric structure (see \cite{Ki87} for more background in physics). Mathematically, string property of a vector bundle $V$ is equivalent to vanishing of characteristic classes $w_{1}(V),w_{2}(V)$ and $p_{1}(V)/2$, where $w_{1}(V),w_{2}(V)$ are Stiefel-Whitney classes and $p_{1}(V)/2$ is half of the first Pontryagin class (see Section 2).

Due to stronger restrictions, search for examples of string manifolds becomes harder and the progress in string category is far less than that in spin category. For instance, explicit group structure and complete invariants of spin bordism were given in \cite{ABP67}, while the same problems remain widely open in string case. Partial results on string bordism group structure can be found in \cite{HR95, MG95} and Witten genus introduced in \cite{Wi88} is conjectured to be part of complete invariants for string bordism.

On the other hand, first introduced by Davis and Januszkiewicz in their pioneering work \cite{DJ91} as a topological generalization of nonsingular projective toric varieties, \emph{quasitoric manifolds} admit equivariant unitary structure and can be applied in various bordism theories. For example, every element in unitary bordism group $\Omega_{*}^{U}$ has a quasitoric representative \cite{BPR07}. And polynomial generators of $\Omega_{*}^{SU} \otimes \Z[\frac{1}{2}]$ (special unitary bordism ring with 2 inverted) can be chosen as Calabi-Yau hypersurfaces in certain quasitoric manifolds \cite{LLP18}. In fact, quasitoric manifolds are sufficient to represent these generators when dimension is greater than 8 \cite{LP16}. 

A quasitoric manifold can be constructed from its orbit polytope and corresponding characteristic matrix, yielding the expression of its algebraic topology data such as cohomology ring and characteristic classes in a combinatorial way. In particular, for any given quasitoric manifold $M$, admitting string structure is equivalent to $w_{2}(M)=p_{1}(M)=0$ since $M$ is orientable and $H^{*}(M;\Z)$ is torsion-free (see Section 2). These properties intrigue an attempt to search and classify string quasitoric manifolds. 

In this article, we apply the most straightforward approach: represent $w_{2}(M)$ (resp. $p_{1}(M)$) as the linear combination of certain basis of $H^{2}(M;\Z)$ (resp. $H^{4}(M;\Z)$). Hence, string property is equivalent to vanishing of corresponding coefficients. It turns out that linear restriction $w_{2}(M)=0$ is only related with parity of elements in the characteristic matrix, and can be formulated in a unified form. However, when it comes to non-linear restriction $p_{1}(M)=0$, one can not get rid of the orbit polytope. This makes it impossible to reach a uniform description of $p_{1}(M)=0$ in combinatorial language via the straightforward approach. As a matter of fact, the existence of quasitoric manifolds over a given simple polytope is related to Buchstaber invariant problem, which remains unsolved in general case (see \cite{Er14} for some partial results). 

Nevertheless, as long as the combinatorial type of the orbit polytope is clear, a complete characterization of $p_{1}(M)=0$ is available, which allows further study at the level of polytope as well as manifold. Furthermore, in some specific cases, partial characterizations can impose effective restrictions on a string quasitoric manifold and its orbit polytope as well. 

\vspace{0.5cm}

Main results of this article include two parts. The first part centers around low dimensional string quasitoric manifolds: 

\begin{thm} \bf{Theorem 1.1}
\it{For $n \leq 3$, an $n$-dimensional simple polytope $P$ can be realized as the orbit polytope of a string quasitoric manifold if and only if $P$ is $n$-colorable.} 
\end{thm}

This theorem is a combination of Proposition \ref{dimP=2 polytope} and \ref{dimP=3 polytope}. In general, $n$-colorability is always a sufficient condition for an $n$-dimensional simple polytope to be realizable as the orbit polytope of a string quasitoric manifold \cite{DJ91}. But as demonstrated in Example \ref{5-colorable}, it is not necessary when $n \geq 4$. In the case of product of two polygons, a more complicated criterion is given in Proposition \ref{dimP=4 polytope}. 

At the level of manifold, 4-dimensional spin quasitoric manifolds are all string. In fact, there is only one string (spin) quasitoric manifold over $2m$-gon up to homeomorphism, namely the \emph{equivariant connected sum} of $m-1$ copies of $\C P^{1} \times \C P^{1}$ (see Remark \ref{dimP=2 counting}). As for 6-dimensional case, general results are not available, but string quasitoric manifolds over prisms can be completely characterized. It is worthwhile to mention that they may not be homeomorphic to a bundle type manifold, i.e., the total space of a fiber bundle, although the orbit polytope $L_{2k}=C_{2}(2k) \times I$ is a Cartesian product (see Example \ref{not bundle type example}). But with the help of \emph{equivariant edge connected sum operation} (see Definition \ref{equivariant edge connected sum definition}), all of them can be constructed from bundle type string quasitoric manifolds. 

\begin{thm} \bf{Theorem \ref{equivariant edge connected sum theorem}}
\it{If a quasitoric manifold $M(L_{2k},\Lambda)$ is string, then there exist quasitoric manifolds $\{M(L_{2k_{i}},\Lambda_{i})\}_{i=1}^{s}$ such that \\
(1) $M(L_{2k_{i}},\Lambda_{i})$ is of bundle type and string for $1 \leq i \leq s$; \\
(2) $M(L_{2k},\Lambda)=M(L_{2k_{1}},\Lambda_{1}) \widetilde{\#^{e}} \cdots \widetilde{\#^{e}} M(L_{2k_{s}},\Lambda_{s})$ up to weakly equivariant homeomorphism.} 
\end{thm}

The second part follows from an observation which imposes strong restrictions on the orbit polytope of a string quasitoric manifold. Particularly, such polytopes must be triangle-free. Based on classification of $n$-dimensional triangle-free simple polytopes with the number of facets no more than $2n+2$ \cite{BB92}, complete characterizations on string quasitoric manifolds over these polytopes are obtained. Illustrative results are listed in Proposition \ref{2n coefficient}-\ref{2n+2B coefficient}. On the other hand, within the range of product of simplices, only cube can be the orbit polytope of a string quasitoric manifold. Further analysis on characteristic matrices leads to additional characterization on the manifold: 

\begin{thm} \bf{Theorem \ref{2n manifold}} 
\it{Every string quasitoric manifold over cube is weakly equivariantly homeomorphic to a Bott manifold. }
\end{thm}

Moreover, if the orbit polytope is the connected sum of a cube and another simple polytope, then for corresponding quasitoric manifolds, string property is compatible with equivariant connected sum operation. 

\begin{thm} \bf{Theorem \ref{equivariant connected sum theorem}} 
\it{$M(I^{n} \# P^{n},\Lambda)$ is string if and only if it is weakly equivariantly homeomorphic to $M(I^{n},\Lambda_{L}) \widetilde{\#} M(P^{n},\Lambda_{R})$ with both $M(I^{n},\Lambda_{L})$ and $M(P^{n},\Lambda_{R})$ string.} 
\end{thm}

Davis and Januszkiewicz \cite{DJ91} also gave a real version of their generalization called \emph{small cover}. Thus, one would expect a real analogue of results above. It should be pointed out that in real case, arguments become much simpler due to calculation in $\Z_{2}$-coefficients. However, parallel results may not be valid, since key problem now lies in the second Stiefel-Whitney class, whose explicit expression is different from that of the first Pontryagin class of a quasitoric manifold. Besides, in certain cases such as product of simplices, partial results of string small cover can be found in literature \cite{CMO17, DU19}. 

\hspace*{\fill}

The rest of this article is organized as follows. In Section 2, we review some basic definitions and properties of string structure and quasitoric manifold. Subsection \ref{quasitoric manifold} mainly follows from the 
collective book \cite[Section 7.3]{BP15}. Section 3 and 4 are devoted to specific characterization of string quasitoric manifolds and their orbit polytopes in certain circumstances. The former focuses on low dimensional cases while the latter deals with cases where orbit polytopes have few facets. In Section 5, we simply list several results and examples in real version without detailed explanation.

\section{Preliminaries} \label{preliminaries}
\setcounter{equation}{0}

\subsection{String structure} \label{string structure}

Let $V$ be a rank $n$ real vector bundle over a manifold $M$. The structural group of $V$ can be reduced to $O(n)$ by Gram-Schmidt orthogonalization and there is a corresponding classifying map $f: M\rightarrow BO(n)$. Additional structures on $V$ lead to further reductions of the structural group and liftings in Whitehead tower:
\[\xymatrix@C=25ex{
& BString(n) \ar[d] \\
& BSpin(n) \ar[d] \\
& BSO(n) \ar[d] \\
M \ar@{-->}"1,2"^{f_{3}} \ar@{-->}"2,2"^{f_{2}} \ar@{-->}"3,2"^{f_{1}} \ar"4,2"^{f} & BO(n)
}
\]
where $Spin(n)$ is a 2-connected cover of $SO(n)$ and $String(n)$ is a 3-connected cover of $Spin(n)$. The existence of liftings $f_{1},f_{2}$ and $f_{3}$ corresponds to orientable, spin and string structures on $V$ with obstruction $w_{1}(V),w_{2}(V)$ and $p_{1}(V)/2$ respectively. A string manifold is characterized by $w_{1}(M)=w_{2}(M)=p_{1}(M)/2=0$, i.e., the string structure on its tangent bundle.

With a more geometric viewpoint, Stolz and Teichner \cite{ST05} discovered that string structure on $M$ is related to fusive spin structure on free loop space $LM$. Relative progress was further achieved by Bunke \cite{Bu11}, Kottke-Melrose \cite{KM13} and Waldorf \cite{Wa15, Wa16}. Moreover, string structure can be regarded as orientability in a generalized cohomology theory called \emph{tmf} (topological modular form), which was developed by Hopkins and Miller via homotopy theoretical methods (see \cite{Ho02} for more details).

\subsection{Quasitoric manifold} \label{quasitoric manifold}

A combinatorial polytope is an equivalent class of convex polytopes with the same face poset. And an orientation of a combinatorial polytope is a permutation equivalent class of its facets. Since only combinatorial polytopes are concerned, we shall use \emph{polytope} throughout this article for brevity.

For an $n$-dimensional polytope $P$, let $f_{i}$ denote the number of its $i$-dimensional faces, then $\boldsymbol{f}(P)=(f_{0},f_{1},\dots,f_{n-1},1)$ is called $f$-vector of $P$ and $\boldsymbol{h}(P)=(h_{0},h_{1},\dots,h_{n-1},h_{n})$ determined by 
\begin{equation} \label{h-vector}
h_{0}s^{n}+h_{1}s^{n-1}+\dots+h_{n}=(s-1)^{n}+f_{n-1}(s-1)^{n-1}+\dots+f_{0}
\end{equation}
is called $h$-vector of $P$. Note that $h_{0}=1$ and $f_{0}=\sum_{i=0}^{n}h_{i}$ by definition. A polytope is called simple if each codimension-$k$ face is the intersection of exactly $k$ facets. When $P$ is simple, there is an additional Dehn-Sommerville relation: $h_{i}=h_{n-i}$ for $0\leq i \leq n$ (see \cite{Gr03} or \cite{Zi98}).

Given an $n$-dimensional simple polytope $P$ with facet set $\cal{F}$$(P)=\{F_{i}\}_{i=1}^{m}$, an integer valued matrix $\Lambda_{n\times m}=(\boldsymbol{\la_{1}},\cdots,\boldsymbol{\la_{m}})\in Mat_{n \times m}(\Z)$ is said to be characteristic if the following nonsingular condition holds: 
\begin{equation}
\forall\ p=F_{j_{1}}\cap \cdots \cap F_{j_{n}} \Rightarrow \mathrm{det}(\Lambda_{p})=\mathrm{det}(\boldsymbol{\la_{j_{1}}},\cdots,\boldsymbol{\la_{j_{n}}})=\pm 1. \label{nonsingular}
\end{equation}
We call $(P, \Lambda_{n\times m})$ a characteristic pair. A $2n$-dimensional quasitoric manifold can be constructed from this pair in a canonical way.

\begin{construction}
Given a characteristic pair $(P, \Lambda_{n\times m})$, for each $p \in P$, there exists a unique face $f(p)$ such that $p$ is in the relative interior of $f(p)$. Suppose that $f(p)=F_{j_{1}}\cap \cdots \cap F_{j_{k}}$, then we can regard the submatrix $(\boldsymbol{\la_{j_{1}}},\cdots,\boldsymbol{\la_{j_{k}}})=(\la_{i,j})\in Mat_{n \times k}(\Z)$ as a linear map from $T^{k}$ to $T^{n}$, sending $(t_{1},\dots,t_{k})$ to $(\prod_{i=1}^{k} t_{i}^{\la_{1,i}},\dots,\prod_{i=1}^{k} t_{i}^{\la_{n,i}})$ and denote by $G(p)$ its image. Define a quasitoric manifold over $P$ as the quotient space 
\begin{equation*}
M(P,\Lambda_{n\times m})=(P\times T^{n})/\sim \quad \mathrm{where}\ (p,\boldsymbol{t_{1}})\sim (p,\boldsymbol{t_{2}})\ \Leftrightarrow \ \boldsymbol{t_{1}^{-1}}\boldsymbol{t_{2}}\in G(p).
\end{equation*}
The free $T^{n}$-action on $P\times T^{n}$ induces an action on $M(P,\Lambda_{n\times m})$, which is free over the interior of $P$ and trivial over vertices of $P$. Furthermore, this action can be locally identified with standard $T^{n}$-action since $M(P,\Lambda_{n\times m})$ is covered by open subsets equivariantly homeomorphic to $\C^{n}=\R_{\geq 0}^{n}\times T^{n}/\sim$.
\end{construction}

In short, a $2n$-dimensional quasitoric manifold is a locally standard $T^{n}$-manifold with quotient space homeomorphic to an $n$-dimensional simple polytope as a manifold with corners. For brevity, the characteristic pair and its corresponding quasitoric manifold will be denoted by $(P,\Lambda)$ and $M(P,\Lambda)$ respectively when no confusion occurs. 

For $2n$-dimensional quasitoric manifolds $M$ and $N$, if there exists $\phi \in \mathrm{Aut}(T^{n})$ and $h \in \mathrm{Homeo}(M,N)$ such that for all $t\in T^{n}$ and $m\in M,\ h(t\cdot m)=\phi(t)\cdot h(m)$, then they are said to be weakly equivariantly homeomorphic. 

There are three group actions on characteristic pairs corresponding to weakly equivariant homeomorphism of quasitoric manifolds: left multiplication by general linear group $\mathrm{GL_{n}}(\Z)$; sign permutation of columns by $\Z_{2}^{m}$ and certain column permutation by $\mathrm{Aut}(\partial P^{*})$ (automorphism group of the dual boundary complex $\partial P^{*}$). $\mathrm{GL_{n}}(\Z)$-action corresponds to the choice of $\phi$ while $\Z_{2}^{m}$-action and $\mathrm{Aut}(\partial P^{*})$-action correspond to the choice of $h$. Two characteristic pairs are called equivalent if one can be obtained from another by a sequence of these three types of group actions. 

\begin{proposition}
\emph{\cite[Proposition 7.3.11]{BP15}} There is a one-to-one correspondence between weakly equivariant homeomorphism classes of quasitoric manifolds and equivalent classes of characteristic pairs.
\end{proposition}

\vspace{0.5cm}

The above-mentioned construction method can be applied to the case of moment-angle manifolds as well. Recall that in an $n$-dimensional simple polytope $P$ with $m$ facets, each point $p$ is contained in the relative interior of a unique face $f(p)=F_{j_{1}}\cap \dots \cap F_{j_{k}}$. Denote $J(p)=\{j_{1},\dots,j_{k}\}$ and $T(p)=\{(t_{1},\dots,t_{m})\in T^{m}\ |\ t_{j}=1\ \mathrm{for}\ j\notin J(p)\}$. Then the moment-angle manifold corresponding to $P$ is defined by 
\begin{equation*}
\mathcal{Z}_{P}=P\times T^{m}/\sim \quad \mathrm{where}\ (p,\boldsymbol{t_{1}})\sim (p,\boldsymbol{t_{2}})\ \Leftrightarrow \ \boldsymbol{t_{1}^{-1}}\boldsymbol{t_{2}}\in T(p).
\end{equation*}
In addition, the characteristic matrix $\Lambda$ can be regarded as a linear map from $T^{m}$ to $T^{n}$ and $\mathrm{Ker}\Lambda$ is isomorphic to $T^{m-n}$ by (\ref{nonsingular}). As a matter of fact, $\mathrm{Ker}\Lambda$ plays a key role in the close relation between $\mathcal{Z}_{P}$ and $M(P,\Lambda)$:

\begin{proposition}
\emph{\cite[Proposition 7.3.13]{BP15}} $\mathrm{Ker}\Lambda$ acts freely and smoothly on $\mathcal{Z}_{P}$ with the quotient $\mathcal{Z}_{P}/\mathrm{Ker}\Lambda$ equivariantly homeomorphic to $M(P,\Lambda)$.
\end{proposition}

On the other hand, the moment-angle manifold $\mathcal{Z}_{P}$ can be constructed directly from the following pullback diagram:
\[
\xymatrix{
\mathcal{Z}_{P} \ar[r] \ar[d] & \mathbb{D}^{m} \ar[d]^{\mu} \\
P \ar@{^(->}[r]^{c_{P}} & I^{m}
}
\]
where $\mu:(z_{1},\dots,z_{m}) \mapsto (|z_{1}|^{2},\dots,|z_{m}|^{2})$ and $c_{P}$ is the cubical subdivision of $P$ (see \cite[Chapter 2]{BP15}). In particular, $\mathcal{Z}_{P}$ admits a smooth structure, inducing the canonical smooth structure on $M(P,\Lambda)$. Furthermore, the proposition below builds an isomorphism between real equivariant bundles over $M(P,\Lambda)$, yielding the equivariant unitary structure on $M(P,\Lambda)$.

\begin{proposition} \label{unitary structure}
\emph{\cite[Theorem 6.6]{DJ91}} Let $\rho_{j}$ denote the equivariant complex line bundle $\mathcal{Z}_{P}\times_{\mathrm{Ker}\Lambda} \C_{j} \rightarrow \mathcal{Z}_{P}/\mathrm{Ker}\Lambda=M(P,\Lambda)$ induced by trivial bundle $\mathcal{Z}_{P}\times \C_{j} \rightarrow \mathcal{Z}_{P}$ viewed as an equivariant bundle with diagonal action of $T^{m}$. Then there is an isomorphism:
\begin{equation*}
TM(P,\Lambda)\oplus \underline{\R}^{2(m-n)} \cong \rho_{1}\oplus \dots \oplus \rho_{m}, 
\end{equation*}
where $\underline{\R}^{2(m-n)}$ is the trivial equivariant bundle of real dimension $2(m-n)$ over $M(P,\Lambda)$.
\end{proposition}

Thus, formulas for characteristic classes of a quasitoric manifold can be deduced: 
\begin{proposition} \label{characteristic class}
\emph{\cite[Corollary 6.7]{DJ91}} Let $v_{j}$ denote the first Chern class of $\rho_{j}$, then 
\begin{equation*}
c(M(P,\Lambda))=\prod_{j=1}^{m}(1+v_{j}) \qquad p(M(P,\Lambda))=\prod_{j=1}^{m}(1+v_{j}^{2}).
\end{equation*}
In particular, $w_{2}(M(P,\Lambda))=\sum\limits_{j=1}^{m}v_{j} \pmod{2}$ and $p_{1}(M(P,\Lambda))=\sum\limits_{j=1}^{m}v_{j}^{2}$.
\end{proposition}

\begin{remark}
For each facet $F_{j}\in \cal{F}$$(P)$, $M_{j}=\pi^{-1}(F_{j})$ is called a characteristic submanifold where $\pi : M(P,\Lambda) \rightarrow P$ is the natural projection. The restriction of $\rho_{j}$ on corresponding characteristic submanifold $M_{j}$ is the normal bundle of embedding $\iota_{j}: M_{j}\hookrightarrow M$. Thus, $v_{j}=c_{1}(\rho_{j})$ is the Poincar\'e dual of $M_{j}$.
\end{remark}

In addition, there is a natural cell decomposition of $M(P,\Lambda)$ given by \emph{Morse-theoretic argument} (see \cite{DJ91}). Consequently, cohomology of $M(P,\Lambda)$ can be characterized as follow: 

\begin{proposition} \label{Betti number}
\emph{\cite[Theorem 3.1]{DJ91}} The integral cohomology groups of $M(P,\Lambda)$ vanish in odd dimensions and are free abelian in even dimensions. The Betti numbers are given by 
\begin{equation*}
\beta^{2i}(M(P,\Lambda))=h_{i}(P), 
\end{equation*}
where $h_{i}(P)$ $(0 \leq i \leq n)$ are components of the $h$-vector of $P$. 
\end{proposition}

\begin{proposition} \label{cohomology ring}
\emph{\cite[Theorem 4.14]{DJ91}} Write $\Lambda=(\la_{i,j})\in Mat_{n \times m}(\Z)$, then the integral cohomology ring of $M(P,\Lambda)$ is given by 
\begin{equation*}
H^{*}(M(P,\Lambda);\Z)\cong \Z[v_{1},\dots,v_{m}]/(\mathcal{I+J}), 
\end{equation*}
where ideal $\mathcal{I}$ is generated by face ring elements $v_{j_{1}}\cdots v_{j_{k}}$ whenever $F_{j_{1}}\cap \dots \cap F_{j_{k}}=\emptyset$ and ideal $\mathcal{J}$ is generated by linear elements $t_{i}=\sum_{j=1}^{m} \la_{i,j}v_{j}$ for $1\leq i \leq n$.
\end{proposition}

Clearly, Stiefel-Whitney classes and Pontryagin classes remain unchanged while Chern classes may vary under weakly equivariant homeomorphism since $\mathrm{GL_{n}}(\Z)$-action, $\Z_{2}^{m}$-action and $\mathrm{Aut}(\partial P^{*})$-action correspond to equivalent expression of ideal $\mathcal{J}$, sign permutation of generators and permutation of generators respectively.

\begin{remark} \label{spin and Betti number count}
By equivalence, a characteristic pair $(P,\Lambda)$ can be refined such that $F_{1}\cap \dots \cap F_{n}$ is a vertex of $P$ (called \emph{initial vertex}) and $\Lambda=[\ \mathrm{I_{n}}\ |\ \Lambda_{*}\ ]$, i.e., $\boldsymbol{\la_{i}}=\boldsymbol{e_{i}}$ for $1\leq i\leq n$. Hence, in refined form, integral basis of $H^{2}(M(P,\Lambda))$ can be chosen as $\{v_{i}\}_{i=n+1}^{m}$ by ideal $\mathcal{J}$ and $w_{2}(M(P,\Lambda))=0$ is equivalent to sum of every column being odd. Consequently, every element in $H^{4}(M(P,\Lambda))$ is an integral linear combination of $\{v_{i}v_{j}\}_{n+1 \leq i, j \leq m}$. Since ideal $\mathcal{I}$ contains $\binom{m}{2}-f_{n-2}(P)$ relations in $H^{4}(M(P,\Lambda))$ and 
\[
\binom{m-n+1}{2}-[\binom{m}{2}-f_{n-2}(P)]=\binom{n}{2}-(n-1)m+f_{n-2}(P)=\beta^{4}(M(P,\Lambda))
\]
by (\ref{h-vector}), such relations are independent to each other, i.e., not a single relation is redundant. When applied to Cartesian product and equivariant connected sum, basis of 4-dimensional cohomology group can be chosen in a canonical way (see Remark \ref{Cartesian product} and Definition \ref{equivariant connected sum definition}). Within the rest of this article, characteristic matrices are assumed to be in refined form unless otherwise stated. 
\end{remark}

\vspace{0.5cm}

Another application of Proposition \ref{unitary structure} is related to signs and weights at fixed points. For a vertex $p=F_{j_{1}}\cap \dots \cap F_{j_{n}}$, the tangent space at $p$ can be decomposed as 
\begin{equation*}
TM|_{p}=\rho_{j_{1}}\oplus \dots \oplus \rho_{j_{n}}|_{p}.
\end{equation*}
Note that the orientation of left hand side is determined by the orientation of $M$ while the orientation of right hand side is determined by orientations of characteristic submanifolds $\{M_{j_{k}}\}_{k=1}^{n}$. The sign $\sigma(p)$ is defined to be $+1$ if these two orientations coincide and $-1$ otherwise. Moreover, tangential representation of $T^{n}$ at $p$ is characterized by weight vectors $\boldsymbol{w_{1}}(p),\cdots,\boldsymbol{w_{n}}(p)$ in $\Z^{n}$: for $\boldsymbol{t}=(e^{2\pi i \varphi_{1}},\dots,e^{2\pi i \varphi_{n}})\in T^{n}$ and $\boldsymbol{z}=(z_{1},\dots,z_{n})\in TM|_{p}$, 
\begin{equation*}
\boldsymbol{t}\cdot \boldsymbol{z}=(e^{2\pi i \langle \boldsymbol{w_{1}}(p),\boldsymbol{\varphi} \rangle} z_{1},\dots,e^{2\pi i \langle \boldsymbol{w_{n}}(p),\boldsymbol{\varphi} \rangle} z_{n}), 
\end{equation*}
where $\boldsymbol{\varphi}=(\varphi_{1},\dots,\varphi_{n})\in \R^{n}$ and $\langle \cdot , \cdot \rangle$ is the standard inner product in $\R^{n}$.

Weights and signs may vary under weakly equivariant homeomorphism and can be calculated directly from a characteristic pair $(P,\Lambda)$:

\begin{proposition}
\emph{\cite[Proposition 7.3.18 and Lemma 7.3.19]{BP15}} \\
(1) For $p=F_{j_{1}}\cap \dots \cap F_{j_{n}}$, weight vectors $\boldsymbol{w_{1}}(p),\cdots,\boldsymbol{w_{n}}(p)$ are the lattice basis conjugate to characteristic vectors, i.e., 
\begin{equation*}
(\boldsymbol{w_{1}}(p),\cdots,\boldsymbol{w_{n}}(p))^{\mathrm{T}}\cdot (\boldsymbol{\la_{j_{1}}},\cdots,\boldsymbol{\la_{j_{n}}})=\mathrm{I_{n}}.
\end{equation*}
(2) Denote $\boldsymbol{a_{j_{k}}}$ as the inward normal vector of facet $F_{j_{k}}$, then 
\begin{equation*}
\sigma(p)=sign(\mathrm{det}(\boldsymbol{\la_{j_{1}}},\cdots,\boldsymbol{\la_{j_{n}}})\cdot \mathrm{det}(\boldsymbol{a_{j_{1}}},\cdots,\boldsymbol{a_{j_{n}}})).
\end{equation*}
\end{proposition}

\begin{example} \label{quasitoric manifold over triangle}
For $P=\Delta^{2}$, the characteristic matrix $\Lambda$ can be refined to $\left(\begin{smallmatrix} 1 & 0 & \delta_{1} \\ 0 & 1 & \delta_{2} \end{smallmatrix}\right)$ with $\delta_{1},\delta_{2}=\pm 1$ by (\ref{nonsingular}). The integral cohomology ring of $M(\Delta^{2},\Lambda)$ is $\Z[v]/\langle v^{3} \rangle$ and $w_{2}(M(\Delta^{2},\Lambda))=v$, $p_{1}(M(\Delta^{2},\Lambda))=3v^{2}$. 

There is only one weakly equivariant homeomorphism class and the corresponding quasitoric manifold is $\C P^{2}$. On the other hand, let $p=F_{2}\cap F_{3}$ (see Figure \ref{triangle with coloring}), then $\boldsymbol{w_{1}}(p)=(-1,1)^{\mathrm{T}}, \boldsymbol{w_{2}}(p)=(1,0)^{\mathrm{T}}$, $\sigma(p)=-1$ when $\delta_{1}=\delta_{2}=1$, and $\boldsymbol{w_{1}}(p)=(-1,1)^{\mathrm{T}}, \boldsymbol{w_{2}}(p)=(-1,0)^{\mathrm{T}}$, $\sigma(p)=+1$ when $\delta_{1}=\delta_{2}=-1$. Similar argument holds for $n$-simplex $\Delta^{n}$.
\end{example}

\begin{figure}[htb]
\centering
\includegraphics[height=4cm]{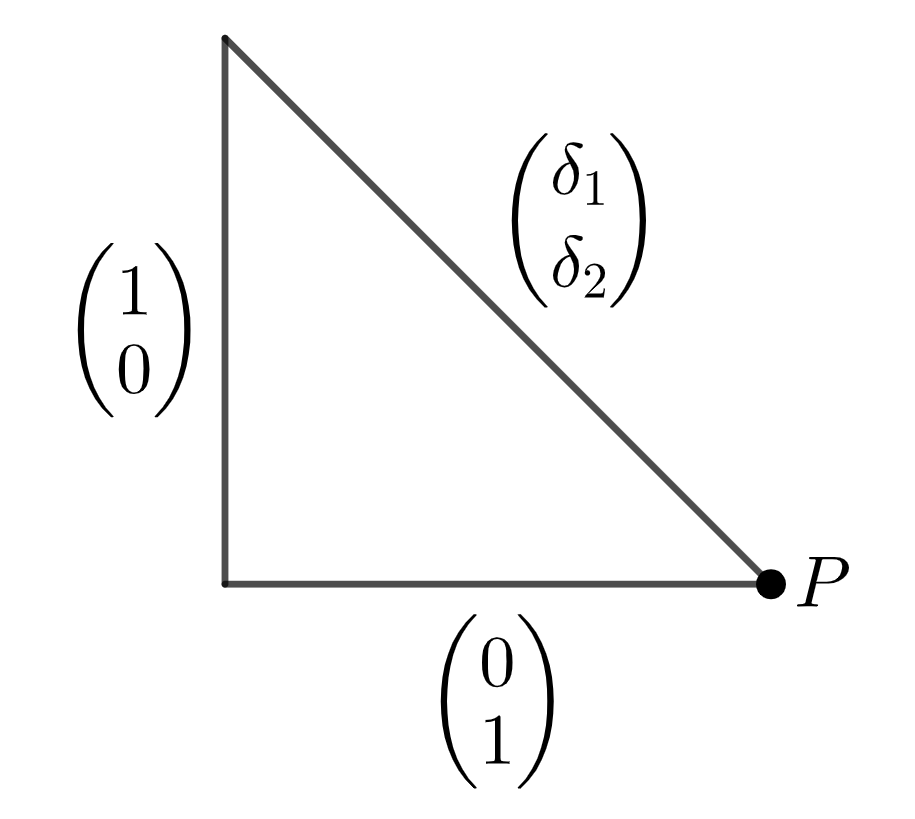}
\caption{$\Delta^{2}$ with characteristic map}
\label{triangle with coloring}
\end{figure}

\section{Low dimensional case}
\setcounter{equation}{0}

\subsection{$\mathrm{dim}P=2$} \label{dimP=2}

For any 4-dimensional quasitoric manifold $M$, $\beta^{4}(M)=1$ and every non-vanishing element in $H^{4}(M)$ can be chosen as the $\Q$-basis. Let $C_{2}(m)$ denote the $m$-gon and label the facets such that $F_{i} \cap F_{j} \neq \emptyset$ if and only if $|i-j| \equiv 1 \pmod{m}$. 

\begin{lemma} \label{dimP=2 coefficient}
Given a quasitoric manifold $M=M(C_{2}(m),\Lambda)$ with $\Lambda=(\la_{i,j})_{2 \times m}$, set $\Delta_{i,j}=\mathrm{det} \left( \begin{smallmatrix} \la_{1,i} & \la_{1,j} \\ \la_{2,i} & \la_{2,j}\end{smallmatrix} \right)$ and $l_{i}=\Delta_{i-1,i} \cdot \Delta_{i,i+1} \cdot \Delta_{i+1,i-1}$ with subscripts taken modulo $m$. Then 
\[
p_{1}(M)=(\sum_{i=1}^{m} l_{i}) v_{1}v_{2}. 
\]
\end{lemma}

\begin{proof}
Based on Proposition \ref{cohomology ring}, for $1 \leq i \leq m$, 
\[
\left\{
\begin{aligned}
& \la_{1,i-1} v_{i-1}v_{i}+\la_{1,i} v_{i}^{2}+\la_{1,i+1} v_{i}v_{i+1}=0; \\
& \la_{2,i-1} v_{i-1}v_{i}+\la_{2,i} v_{i}^{2}+\la_{2,i+1} v_{i}v_{i+1}=0;
\end{aligned}
\right.
\]
with subscripts taken modulo $m$. Since (\ref{nonsingular}) requires that $\Delta_{i-1,i}, \Delta_{i,i+1}= \pm 1$, equations above are equivalent to 
\[
\left\{
\begin{aligned}
& \Delta_{i-1,i} v_{i-1}v_{i}=\Delta_{i,i+1} v_{i}v_{i+1}; \\
& v_{i}^{2}=l_{i} \cdot \Delta_{i,i+1} v_{i}v_{i+1}. 
\end{aligned}
\right.
\]
This leads to the following expression: 
\[
p_{1}(M)= \sum_{i=1}^{m} v_{i}^{2}= \sum_{i=1}^{m} l_{i} \cdot \Delta_{i,i+1} v_{i}v_{i+1}= (\sum_{i=1}^{m} l_{i}) \cdot \Delta_{1,2} v_{1}v_{2}= (\sum_{i=1}^{m} l_{i}) v_{1}v_{2}.  
\]
\end{proof}

As shown in Example \ref{quasitoric manifold over triangle}, quasitoric manifolds over $C_{2}(3)$ can not be spin, let alone string. As a matter of fact, this can be generalized to the case of arbitrary odd-gon. 

\begin{lemma} \label{dimP=2 main}
For a quasitoric manifold $M$ over $C_{2}(m)$: 
\begin{enumerate}[(1)]
\item $M$ is string $\Leftrightarrow$ $M$ is spin; 
\item $M$ is spin $\Rightarrow$ $M$ bounds in $\Omega_{4}^{SO}$; 
\item $M$ bounds in $\Omega_{4}^{O}$ $\Leftrightarrow$ $m \equiv 0 \pmod{2}$. 
\end{enumerate}
\end{lemma}

\begin{proof}
Since $M$ always admits a non-trivial circle action, spin property induces the vanishing of $\hat{A}(M)$ by a classical result of Atiyah and Hirzebruch \cite{AH70}. On the other hand, $\hat{A}(M)=-\frac{1}{24}p_{1}(M)$ when $\mathrm{dim}M=4$ \cite{Hi66}. Therefore, every 4-dimensional spin quasitoric manifold is indeed string, i.e., (1) is valid. 

Now suppose $M$ is spin, then characteristic numbers corresponding to $w_{1}^{4}$, $w_{1}^{2}w_{2}$, $w_{1}w_{3}$, $w_{2}^{2}$ and $p_{1}$ vanish. On the other hand, Todd genus corresponds to $\frac{1}{12}(c_{1}^{2}+c_{2})$ in dimension 4. And it is integral by Hirzebruch-Riemann-Roch theorem \cite{Hi66}, leading to $w_{4} \equiv c_{2} \equiv c_{1}^{2} \equiv p_{1} \equiv 0 \pmod{2}$. Thus, spin property of $M$ induces its boundness in $\Omega_{4}^{SO}$. 

Combine the argument above with Lemma \ref{dimP=2 coefficient}, the proof of (3) boils down to verification of the following formula: 
\begin{equation} \label{mod 2}
\sum_{j=1}^{m} l_{j} \equiv m \pmod{2}. 
\end{equation}

Let $\mu_{i,j}$ characterize the parity of $\la_{i,j}$, i.e., $\mu_{i,j} \equiv \la_{i,j} \pmod{2}$ for $1 \leq i \leq 2, 1 \leq j \leq m$. If $\mu_{1,j}+\mu_{2,j}=1$ is valid for all $j$, then by (\ref{nonsingular}), 
\[
\overline{\Lambda}=(\mu_{i,j})_{2 \times m}=\begin{pmatrix}
1 & 0 & \cdots & 1 & 0 \\
0 & 1 & \cdots & 0 & 1 
\end{pmatrix}. 
\]
Clearly, (\ref{mod 2}) holds in this case since $l_{j} \equiv m \equiv 0 \pmod{2}$ $(1 \leq j \leq m)$. Otherwise, by (\ref{nonsingular}) again, there are no adjacent columns with both column sums equal to 2. Consider the following type of block within $\overline{\Lambda}$: $B=(\mu_{i,j})_{1 \leq i \leq 2, s_{1} \leq j \leq s_{2}}$ such that 
\[
\mu_{1,j}+\mu_{2,j}=
\left\{
\begin{aligned}
& 1 & \qquad j=s_{1}-1, s_{1}, s_{1}+2, \dots, s_{2}-2, s_{2}, s_{2}+1; \\
& 2 & j=s_{1}+1, s_{1}+3, \dots, s_{2}-3, s_{2}-1, 
\end{aligned}
\right. 
\]
where all subscripts are taken modulo $m$. It is evident that 
\[
l_{j} \equiv 
\left\{
\begin{aligned}
& 0 \pmod{2} & \qquad j \in \mathcal{N}_{1}\setminus \{s_{1}, s_{2}\} \text{ or } j \in \mathcal{N}_{2}; \\
& 1 \pmod{2} & j \in \{s_{1}, s_{2}\} \text{ or } j \in \mathcal{N}_{3},
\end{aligned}
\right.
\]
where $\mathcal{N}_{1}=\{s_{1} \leq j \leq s_{2}\ |\ \mu_{1,j}+\mu_{2,j}=1\}$, $\mathcal{N}_{2}=\{s_{1} \leq j \leq s_{2}\ |\ \mu_{1,j}+\mu_{2,j}=2, \mu_{1,j-1}=\mu_{1,j+1} \}$ and $\mathcal{N}_{3}=\{s_{1} \leq j \leq s_{2}\ |\ \mu_{1,j}+\mu_{2,j}=2, \mu_{1,j-1} \neq \mu_{1,j+1} \}$. Thus, $\sum_{j=s_{1}}^{s_{2}} l_{j} \equiv |\mathcal{N}_{3}| \pmod{2}$. If $|\mathcal{N}_{3}|$ is even, then column $s_{1}$ and $s_{2}$ are identical. In this case, the parity of both $\sum_{j=1}^{m} l_{j}$ and $m$ remain unchanged after deleting the block $B$ except for column $s_{1}$. If $|\mathcal{N}_{3}|$ is odd, then column $s_{1}-1$ and $s_{2}$ are identical. In this case, the parity of both $\sum_{j=1}^{m} l_{j}$ and $m$ change after deleting the block $B$. All columns with column sum 2 can be removed from $\overline{\Lambda}$ by such deletion process, leading to the validity of (\ref{mod 2}). 
\end{proof}

\begin{proposition} \label{dimP=2 manifold}
A quasitoric manifold $M(C_{2}(m),\Lambda)$ is string if and only if 
\[
\la_{1,i}+\la_{2,i} \equiv 1 \pmod{2} \qquad 3 \leq i \leq m. 
\]
\end{proposition}

\begin{proposition} \label{dimP=2 polytope}
$P=C_{2}(m)$ can be realized as the orbit polytope of a string (spin) quasitoric manifold if and only if $P$ is 2-colorable, i.e., $m \equiv 0 \pmod{2}$. 
\end{proposition}

\begin{remark}
$M(C_{2}(4),\Lambda_{1})$ with $\Lambda_{1}=\left( \begin{smallmatrix} 1 & 0 & 1 & 1 \\ 0 & 1 & 0 & 1 \end{smallmatrix} \right)$ is not spin but bounds in $\Omega_{4}^{SO}$; $M(C_{2}(4),\Lambda_{2})$ with $\Lambda_{2}=\left( \begin{smallmatrix} 1 & 0 & 1 & 2 \\ 0 & 1 & 1 & 1 \end{smallmatrix} \right)$ bounds in $\Omega_{4}^{O}$ but does not bound in $\Omega_{4}^{SO}$. Therefore, relationships in Proposition \ref{mod 2} are strict. 
\end{remark}

\begin{remark} \label{dimP=2 counting}
Orlik and Raymond \cite{OR70} showed that up to homeomorphism, every 4-dimensional quasitoric manifold is the equivariant connected sum (see Definition \ref{equivariant connected sum definition}) of $\C P^{1} \times \C P^{1}, \C P^{2}$ and $\overline{\C P^{2}}$. This leads to the following counting result: there is exactly one homeomorphism class of string quasitoric manifold over $C_{2}(2m_{0})$ for each $m_{0} \geq 2$, namely $\widetilde{\#}_{m_{0}-1} (\C P^{1} \times \C P^{1})$. On the other hand, there are countably many weakly equivariant homeomorphism classes since characteristic matrices 
\[
\Lambda=
\left(
\begin{array}{cc|cc|cc|c|cc}
1 & 0 & 1 & 2a_{1} & 1 & 2a_{2} & \cdots & 1 & 2a_{m_{0}-1} \\
0 & 1 & 0 & 1 & 0 & 1 & \cdots & 0 & 1
\end{array}
\right)
\]
are not equivalent to each other for general integral parameters $a_{1}, \dots, a_{m_{0}-1}$. 
\end{remark}

\subsection{$\mathrm{dim}P=3$} \label{dimP=3}

For a 3-dimensional simple polytope $P$ with $m$ facets $(m \geq 4)$, $\boldsymbol{f}(P)=(2m-4, 3m-6, m, 1)$ and $\boldsymbol{h}(P)=(1, m-3, m-3, 1)$. Unlike 2-dimensional case, the same $f$-vector ($h$-vector) may correspond to different combinatorial types. Therefore, a general characterization for 6-dimensional string quasitoric manifolds analogous to Proposition \ref{dimP=2 manifold} does not exist. However, the following proposition on the orbit polytope serves as an analogue to Proposition \ref{dimP=2 polytope}: 

\begin{proposition} \label{dimP=3 polytope}
A 3-dimensional simple polytope $P$ can be realized as the orbit polytope of a string quasitoric manifold if and only if $P$ is 3-colorable. 
\end{proposition}

\begin{proof}
The ``if" part follows directly from the fact that $n$-coloring of an $n$-dimensional simple polytope yields a quasitoric manifold with trivial tangent bundle (called \emph{pull back of the linear model} in \cite{DJ91}). 

Now suppose $P$ is not 3-colorable. By a well-known result of Joswig \cite{Jo02}, $P$ has at least one facet $F$ with odd number of edges. 

If $F$ is a triangle, then let $w=F \cap F_{a} \cap F_{b}$ and $F_{c}$ be the other facet adjacent to $F$. Fix $w$ as the initial vertex, then $v_{c}^{2}$ does not appear in any relation in $H^{4}(M)$ for every quasitoric manifold over $P$. Thus, both $v_{c}^{2}$ and its corresponding coefficient in the expression of $p_{1}(M)$ are nonzero (see Key Observation in Section \ref{few facets case} for the general case). In conclusion, $P$ can not be the orbit polytope of a string quasitoric manifold. 

\begin{figure}[htb]
\centering
\includegraphics[height=4cm]{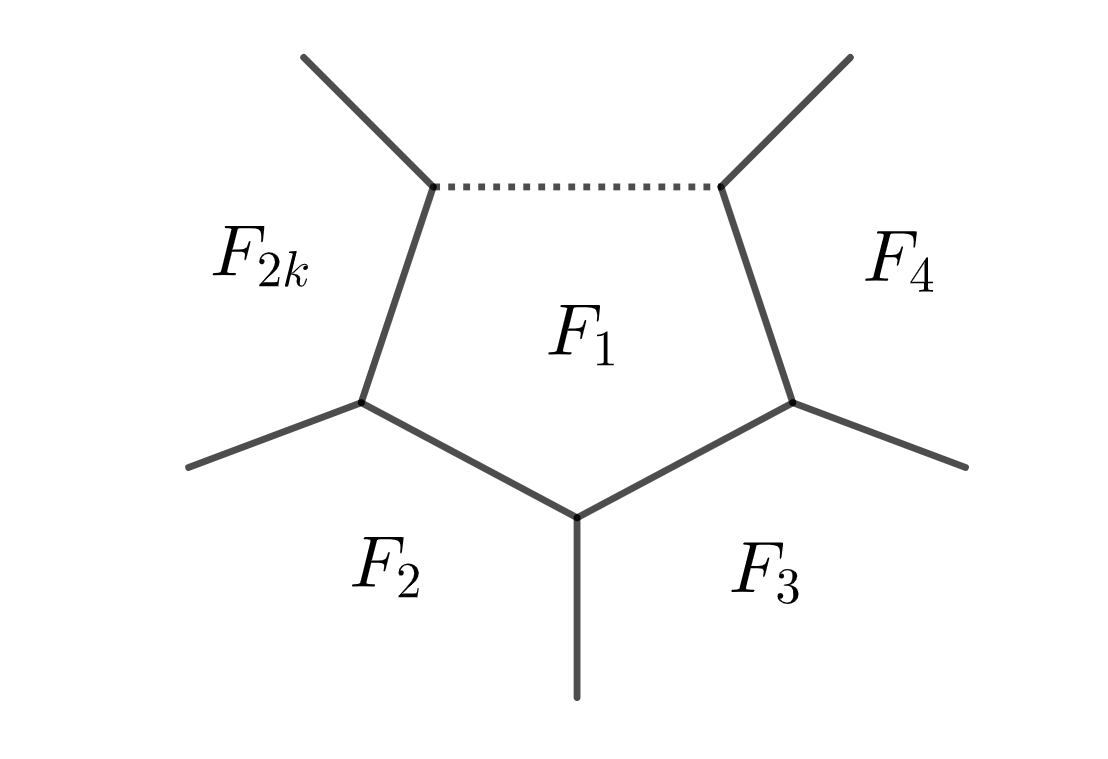}
\caption{Labels around a facet with odd number of edges}
\label{oddpolygon}
\end{figure}

If $F=C_{2}(2k-1)$ $(k \geq 3)$, then let $F_{1}=F$ and label the facets adjacent to $F$ as shown in Figure \ref{oddpolygon}. For $2 \leq i < j \leq 2k$ with $j-i \neq 1, 2k-2$, set $d(i,j)=j-i$ and $d(j,i)=2k-1+i-j$ whenever $F_{i} \cap F_{j} \neq \emptyset$. If such pair of $i,j$ does not exist, keep the labels unchanged. Otherwise, without loss of generality, suppose $d(i_{0},j_{0})$ reaches the minimum value of all such $d(i,j)$ and $d(j,i)$, then for $2 \leq t \leq 2k$, relabel $F_{t}$ as $F_{t-i_{0}+1}$ with subscripts taken modulo $2k-1$. It should be pointed out that similar to argument in the paragraph above, one can deduce $d(i_{0},j_{0}) \geq 3$ from string property. In this way, $F_{2} \cap F_{t}=\emptyset$ for $4 \leq t \leq j_{0}-i_{0}+1$ by minimum value condition. Meanwhile, $F_{2} \cap F_{t}=\emptyset$ for $j_{0}-i_{0}+2 \leq t \leq 2k-1$ by Jordan Curve Theorem since $F_{1}, F_{j_{0}-i_{0}+1}$ and $F_{2k}$ form a 3-belt of $P$. Parallel results are valid for $F_{3}$. To summarize, $\{4 \leq t \leq 2k\ |\ v_{2}v_{t} \neq 0\}=\{2k\}$ and $\{4 \leq t \leq 2k\ |\ v_{3}v_{t} \neq 0\}=\{4\}$. 

Label other facets of $P$ arbitrarily and suppose string property is satisfied for $M=M(P,\Lambda)$ with 
\[
\Lambda=
\left(
\begin{array}{c|ccccc|c}
1 & 0 & 0 & \la_{1,4} & \cdots & \la_{1,2k} & \cdots \\
\hline
0 & 1 & 0 & 1 & \cdots & \la_{2,2k} & \cdots \\
0 & 0 & 1 & \la_{3,4} & \cdots & 1 & \cdots 
\end{array}
\right). 
\]
Similar to Lemma \ref{dimP=2 coefficient}, for $2 \leq i, j \leq 2k$, let $\Delta_{i,j}=\mathrm{det} \left( \begin{smallmatrix} \la_{2,i} & \la_{2,j} \\ \la_{3,i} & \la_{3,j}\end{smallmatrix} \right)$ and $l_{i}=\Delta_{i-1,i} \cdot \Delta_{i,i+1} \cdot \Delta_{i+1,i-1}$ with subscripts taken modulo $2k-1$. Then the following relations can be deduced from cohomology ring structure of $M$: 
\[
\left\{
\begin{aligned}
& \Delta_{i-1,i} v_{i-1}v_{i}=\Delta_{i,i+1} v_{i}v_{i+1}+\cdots & \qquad 5 \leq i \leq 2k-1; \\
& v_{i}^{2}=l_{i} \cdot \Delta_{i,i+1} v_{i}v_{i+1}+\cdots & 4 \leq i \leq2k-1; \\
& v_{2k}^{2}=l_{2k} \cdot \Delta_{2k-1,2k} v_{2k-1}v_{2k}+\cdots. 
\end{aligned}
\right.
\]
Let $\rho_{i}=\sum_{t=1}^{3} \la_{t,i}^{2}+1$, $\rho_{i,j}=\rho_{j,i}=2 \sum_{t=1}^{3} \la_{t,i}\la_{t,j}$ for $1 \leq i < j \leq m$ and choose $v_{2k-1}v_{2k}$ to be one of the $\Q$-basis of $H^{4}(M)$. Then in the expression of $p_{1}(M)$, the corresponding coefficient 
\begin{equation} \label{cyclic coefficient}
c_{2k-1,2k}=\Delta_{2k-1,2k} (\sum_{i=4}^{2k} l_{i} \rho_{i}+\sum_{i=4}^{2k-1} \Delta_{i,i+1} \rho_{i,i+1}). 
\end{equation}
(\ref{cyclic coefficient}) can be further modified as 
\begin{equation} \label{cyclic coefficient 2}
c_{2k-1,2k}=\Delta_{2k-1,2k} \sum_{i=2}^{2k} (l_{i} \rho_{i}+ \Delta_{i,i+1} \rho_{i,i+1})
\end{equation}
with indices taken modulo $2k-1$ since 
\begin{align*}
& \Delta_{2k,2}\rho_{2k,2}+l_{2}\rho_{2}+\Delta_{2,3}\rho_{2,3}+l_{3}\rho_{3}+\Delta_{3,4}\rho_{3,4} \\
= & -2\la_{2,2k}\la_{3,2k}+2\la_{2,2k}\la_{3,2k}+0+2\la_{2,4}\la_{3,4}-2\la_{2,4}\la_{3,4} \\
= &\ 0. 
\end{align*}
Combine (\ref{cyclic coefficient 2}) with algebraic lemma \ref{algebraic lemma 1} and \ref{algebraic lemma 2} below, $c_{2k-1,2k} \equiv 4 \pmod{8}$. In particular, $c_{2k-1,2k} \neq 0$, leading to contradiction. 
\end{proof}

\begin{lemma} \label{algebraic lemma 1}
With indices taken modulo $2k-1$, 
\begin{equation} \label{mod 8}
\sum_{i=2}^{2k} [l_{i}(\la_{1,i}^2+1)+2\Delta_{i,i+1}\la_{1,i}\la_{1,i+1}] \equiv 4 \pmod{8}. 
\end{equation}
\end{lemma}

\begin{proof}
Let $\mu_{i,j}$ characterize the parity of $\la_{i,j}$, i.e., $\mu_{i,j} \equiv \la_{i,j} \pmod{2}$ for $1 \leq i \leq 3, 2 \leq j \leq 2k$. Set $\mathcal{E}=\{2 \leq j \leq 2k\ |\ \mu_{1,j}=0\}$ and $\mathcal{O}=\{2 \leq j \leq 2k\ |\ \mu_{1,j}=1\}$. By (\ref{nonsingular}) and spin property, $\mathcal{O}$ does not contain adjacent indices. Moreover, submatrix $(\mu_{i,j})_{1 \leq i \leq 3, 2\leq j \leq 2k}$ can be divided into two types of blocks: $B_{E}=(\mu_{i,j})_{1 \leq i \leq 3, s_{1} \leq j \leq s_{2}}$ such that 
\[
s_{1}-1, s_{1}, \dots, s_{2}, s_{2}+1 \in \mathcal{E}, 
\]
and $B_{O}=(\mu_{i,j})_{1 \leq i \leq 3, t_{1} \leq j \leq t_{2}}$ such that
\[
\left\{
\begin{aligned}
& t_{1}-1, t_{1}, t_{1}+2, \dots, t_{2}-2, t_{2}, t_{2}+1 \in \mathcal{E}; \\
& t_{1}+1, t_{1}+3, \dots, t_{2}-3, t_{2}-1 \in \mathcal{O}. 
\end{aligned}
\right.
\]

In the first type, $l_{i} \equiv \la_{1,i} \equiv \la_{1,i+1} \equiv 0 \pmod{2}$ for $s_{1} \leq i \leq s_{2}$. Thus, 
\[
\sum_{i=s_{1}}^{s_{2}} [l_{i}(\la_{1,i}^2+1)+2\Delta_{i,i+1}\la_{1,i}\la_{1,i+1}] \equiv \sum_{i=s_{1}}^{s_{2}} l_{i} \pmod{8}. 
\]
In the second type, let $\mathcal{E}_{O}=\{t_{1} \leq j \leq t_{2}\ |\ \mu_{1,j}=0\}$ and $\mathcal{O}_{O}=\{t_{1} \leq j \leq t_{2}\ |\ \mu_{1,j}=1\}$, then with mod 8 operation, 
\begin{align*}
& \sum_{i=t_{1}}^{t_{2}} [l_{i}(\la_{1,i}^2+1)+2\Delta_{i,i+1}\la_{1,i}\la_{1,i+1}] \\
\equiv & \sum_{i \in \mathcal{O}_{O}} 2l_{i}+\sum_{i \in \mathcal{E}_{O}} l_{i}+\la_{1,t_{1}}^{2}+\la_{1,t_{2}}^{2}+2\la_{1,t_{1}}+2\la_{1,t_{2}} \\
\equiv & \sum_{i \in \mathcal{O}_{O}} 2l_{i}+\sum_{i \in \mathcal{E}_{O}} l_{i}. 
\end{align*}
In summary, (\ref{mod 8}) is equivalent to 
\begin{equation} \label{simplified mod 8}
\sum_{i \in \mathcal{E}} l_{i}+\sum_{i \in \mathcal{O}} 2l_{i} \equiv 4 \pmod{8}. 
\end{equation}

On the other hand, it follows from definition that 
\begin{align*}
l_{i} = &\ (\la_{2,i-1}\la_{3,i}-\la_{2,i}\la_{3,i-1}) (\la_{2,i}\la_{3,i+1}-\la_{2,i+1}\la_{3,i}) (\la_{2,i+1}\la_{3,i-1}-\la_{2,i-1}\la_{3,i+1}) \\
= &\ \la_{2,i-1}\la_{3,i-1}(\la_{2,i}^{2}\la_{3,i+1}^{2}-\la_{2,i+1}^{2}\la_{3,i}^{2})+\la_{2,i}\la_{3,i}(\la_{2,i+1}^{2}\la_{3,i-1}^{2}-\la_{2,i-1}^{2}\la_{3,i+1}^{2}) \\
& +\la_{2,i+1}\la_{3,i+1}(\la_{2,i-1}^{2}\la_{3,i}^{2}-\la_{2,i}^{2}\la_{3,i-1}^{2}). 
\end{align*}
In this way, one can rewrite the left hand side of (\ref{simplified mod 8}) to be $\sum_{i=2}^{2k} A_{i} \la_{2,i}\la_{3,i}$, where 
\begin{align*}
A_{i} = &\ (\la_{2,i-2}^{2}\la_{3,i-1}^{2}-\la_{2,i-1}^{2}\la_{3,i-2}^{2})(1+\mu_{1,i-1})+(\la_{2,i+1}^{2}\la_{3,i-1}^{2}-\la_{2,i-1}^{2}\la_{3,i+1}^{2}) \\
&\ (1+\mu_{1,i})+(\la_{2,i+1}^{2}\la_{3,i+2}^{2}-\la_{2,i+2}^{2}\la_{3,i+1}^{2})(1+\mu_{1,i+1}). 
\end{align*}

For each column $i$ in $B_{E}$, $A_{i} \equiv 1+0-1 \equiv 0 \pmod{4}$ and $A_{i} \la_{2,i}\la_{3,i} \equiv 0 \pmod{8}$. As for columns in $B_{O}$: 

If $i=t_{1}$, then
\[
(\mu_{s,t})_{1 \leq s \leq 3, i-2 \leq t \leq i+2}=
\begin{pmatrix}
x & 0 & 0 & 1 & 0 \\
1 & 0 & 1 & 1 & y \\
x & 1 & 0 & 1 & 1-y 
\end{pmatrix}
\text{ or }
\begin{pmatrix}
x & 0 & 0 & 1 & 0 \\
x & 1 & 0 & 1 & y \\
1 & 0 & 1 & 1 & 1-y 
\end{pmatrix}
\]
with $x, y \in \{0, 1\}$ and $A_{i} \equiv 1+1\pm 2 \equiv 0 \pmod{4}$. The same result holds for $i=t_{2}$ in a similar manner. If $i \in \mathcal{E}_{O}\setminus \{t_{1}, t_{2}\}$, then
\[
(\mu_{s,t})_{1 \leq s \leq 3, i-2 \leq t \leq i+2}=
\begin{pmatrix}
0 & 1 & 0 & 1 & 0 \\
x & 1 & a & 1 & y \\
1-x & 1 & 1-a & 1 & 1-y 
\end{pmatrix}
\]
with $a, x, y \in \{0, 1\}$ and $A_{i} \equiv 2+0+2 \equiv 0 \pmod{4}$. Thus, $A_{i}\la_{2,i}\la_{3,i} \equiv 0 \pmod{8}$ is valid for all $i \in \mathcal{E}_{O}$. 

In addition, if $i=t_{1}+1$ and $\mu_{2,i-1}=\mu_{2,i+1}$, then 
\[
(\mu_{s,t})_{1 \leq s \leq 3, i-2 \leq t \leq i+2}=
\begin{pmatrix}
0 & 0 & 1 & 0 & 1 \\
1-a & a & 1 & a & 1 \\
a & 1-a & 1 & 1-a & 1 
\end{pmatrix}
\]
with $a \in \{0, 1\}$ and $A_{i} \equiv \la_{2,i+1)}^{2}+\la_{3,i+1}^{2}-1 \pmod{8}$. If $i=t_{1}+1$ and $\mu_{2,i-1} \neq \mu_{2,i+1}$, then 
\[
(\mu_{s,t})_{1 \leq s \leq 3, i-2 \leq t \leq i+2}=
\begin{pmatrix}
0 & 0 & 1 & 0 & 1 \\
1-a & a & 1 & 1-a & 1 \\
a & 1-a & 1 & a & 1 
\end{pmatrix}
\]
with $a \in \{0, 1\}$ and $A_{i} \equiv (\la_{2,i+1}^{2}+\la_{3,i+1}^{2}-1)+4 \pmod{8}$. Parallel results hold for $i=t_{2}-1$. Similarly, for $i \in \mathcal{O}_{O}\setminus \{t_{1}+1, t_{2}-1\}$, 
\[
A_{i} \equiv 
\left\{
\begin{aligned}
& (\la_{2,i-1}^{2}+\la_{3,i-1}^{2}-1)+(\la_{2,i+1}^{2}+\la_{3,i+1}^{2}-1) \pmod{8} & \qquad \mu_{2,i-1}=\mu_{2,i+1}; \\
& (\la_{2,i-1}^{2}+\la_{3,i-1}^{2}-1)+(\la_{2,i+1}^{2}+\la_{3,i+1}^{2}-1)+4 \pmod{8} & \mu_{2,i-1} \neq \mu_{2,i+1}. 
\end{aligned}
\right. 
\]
Note that for each $i \in \mathcal{O}_{O}$, $A_{i} \la_{2,i}\la_{3,i} \equiv A_{i} \pmod{8}$ since $A_{i} \equiv 0 \pmod{4}$. 

In conclusion, with mod 8 operation and $\mathcal{O'}=\{i \in \mathcal{O}| \mu_{2,i-1} \neq \mu_{2,i+1}\}$, 
\begin{align*}
\sum_{i=2}^{2k} A_{i} \la_{2,i} \la_{3,i} & = \sum_{i \in \mathcal{E}} A_{i} \la_{2,i} \la_{3,i}+\sum_{i \in \mathcal{O}} A_{i} \la_{2,i} \la_{3,i} \\
& \equiv \sum_{i \in \mathcal{O}} A_{i} \\
& \equiv \sum_{i \in \mathcal{O'}} 4 \\
& \equiv 4. 
\end{align*}
Here the last equivalence is induced by the fact $\la_{2,2}=\la_{3,3}=1$ and $\la_{2,3}=\la_{3,2}=0$. 
\end{proof}

\begin{lemma} \label{algebraic lemma 2}
With indices taken modulo $2k-1$, 
\begin{equation} \label{vanish}
\sum_{i=2}^{2k} [l_{i}(\la_{2,i}^2+\la_{3,i}^2)+2\Delta_{i,i+1}(\la_{2,i}\la_{2,i+1}+\la_{3,i}\la_{3,i+1})]=0. 
\end{equation}
\end{lemma}

\begin{proof}
By direct computation, 
\begin{align*}
&\ l_{i}(\la_{2,i}^2+\la_{3,i}^2)+\Delta_{i,i+1}(\la_{2,i}\la_{2,i+1}+\la_{3,i}\la_{3,i+1}) \\
= & -\Delta_{i-1,i} \cdot \Delta_{i,i+1} \cdot [(\la_{2,i}^{2}+\la_{3,i}^{2})(\la_{2,i-1} \la_{3,i+1}-\la_{2,i+1} \la_{3,i-1}) \\
& -(\la_{2,i} \la_{2,i+1}+\la_{3,i} \la_{3,i+1})(\la_{2,i-1} \la_{3,i}-\la_{2,i} \la_{3,i-1})] \\
= & -\Delta_{i-1,i} \cdot \Delta_{i,i+1} \cdot (\la_{2,i-1} \la_{2,i} \Delta_{i,i+1}+\la_{3,i-1} \la_{3,i} \Delta_{i,i+1}) \\
= & -\Delta_{i-1,i} \cdot (\la_{2,i-1} \la_{2,i}+\la_{3,i-1} \la_{3,i}). 
\end{align*}
Therefore, the $i^{th}$ term in left hand side of (\ref{vanish}) equals to 
\[
-\Delta_{i-1,i} \cdot (\la_{2,i-1} \la_{2,i}+\la_{3,i-1} \la_{3,i})+\Delta_{i,i+1}(\la_{2,i}\la_{2,i+1}+\la_{3,i}\la_{3,i+1}), 
\]
and taking cyclic sum yields (\ref{vanish}). 
\end{proof}

\vspace{0.5cm}

As for concrete examples, combine Theorem \ref{dimP=3 polytope} with arguments in Section \ref{dimP=2}, further discussions on string quasitoric manifolds over prism $L_{2k}=C_{2}(2k) \times I$ $(k \geq 2)$ are available. Label the facets of $L_{2k}$ such that $F_{1}$ (resp. $F_{2k+2}$) is the top (resp. bottom) facet and for $2 \leq i < j \leq 2k+1$, $F_{i} \cap F_{j} \neq \emptyset$ if and only if $j-i=1$ or $2k-1$ (see Figure \ref{L2k and its dual with label}). Let the initial vertex be $v=F_{1} \cap F_{2} \cap F_{3}$, then $\Lambda$ can be refined to 
\[
\left(
\begin{array}{ccc|cccc|c}
1 & 0 & 0 & \la_{1,4} & \la_{1,5} & \cdots & \la_{1,2k+1} & 1 \\
0 & 1 & 0 & 1 & \la_{2,5} & \cdots & \la_{2,2k+1} & \la_{2,2k+2} \\
0 & 0 & 1 & \la_{3,4} & \la_{3,5} & \cdots & 1 & \la_{3,2k+2} \\
\end{array}
\right).
\]
Fix the corresponding basis of $H^{2}(M(L_{2k},\Lambda))$ as $\{v_{i}\}_{i=4}^{2k+2}$, then $\{v_{i}v_{2k+2}\}_{4 \leq i \leq 2k+1}$ along with $v_{k+2}v_{k+3}$ form a basis of $H^{4}(M(L_{2k},\Lambda))$. In this way, $p_{1}(M(L_{2k},\Lambda))$ can be expressed as $c_{k+2,k+3} v_{k+2}v_{k+3}+\sum_{i=4}^{2k+1} c_{i,2k+2} v_{i}v_{2k+2}$. 

\begin{figure}[htb]
\begin{minipage}[t]{0.45\linewidth}
\centering
\includegraphics[height=6cm]{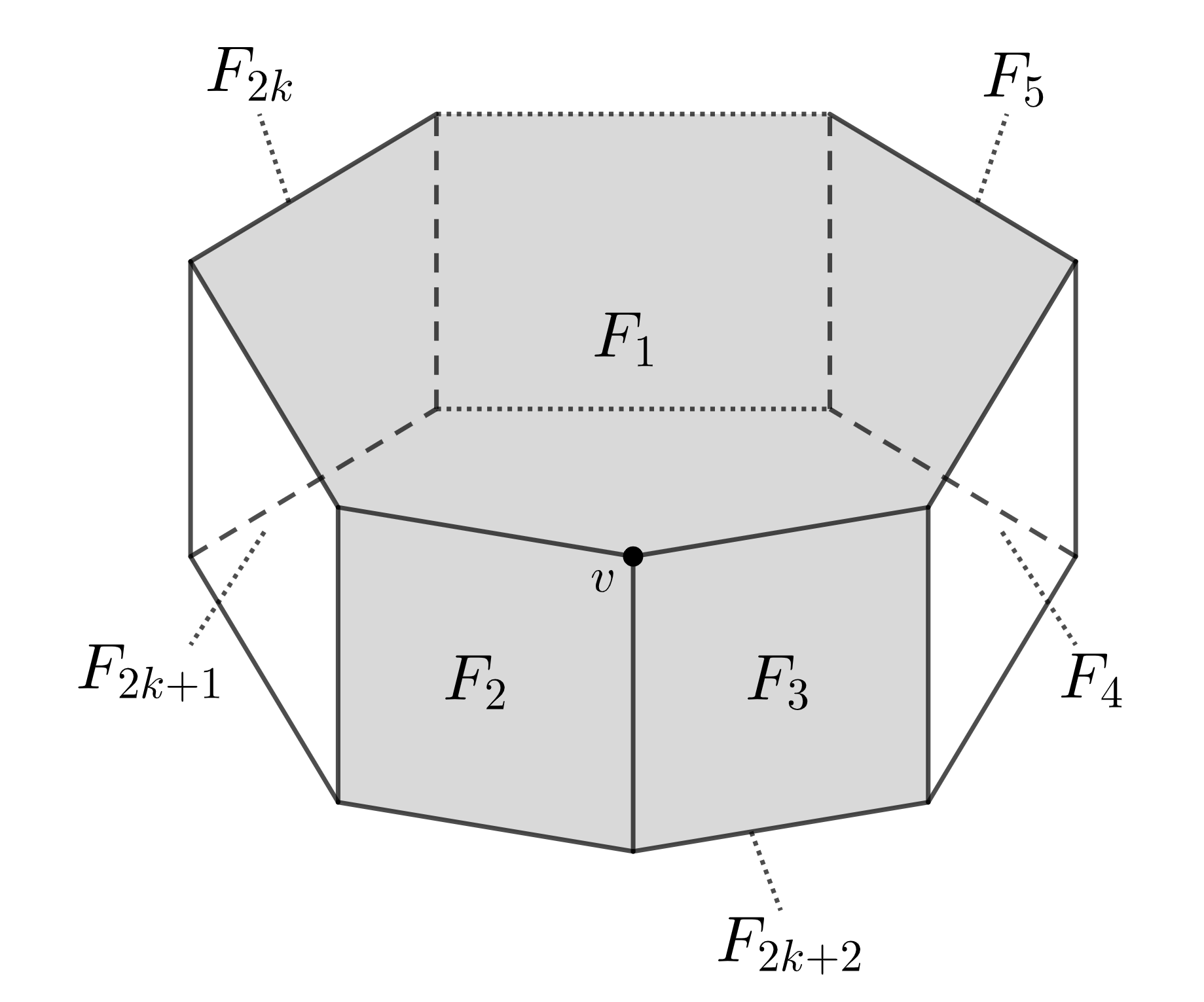}
\end{minipage}
\begin{minipage}[t]{0.45\linewidth}
\centering
\includegraphics[height=6cm]{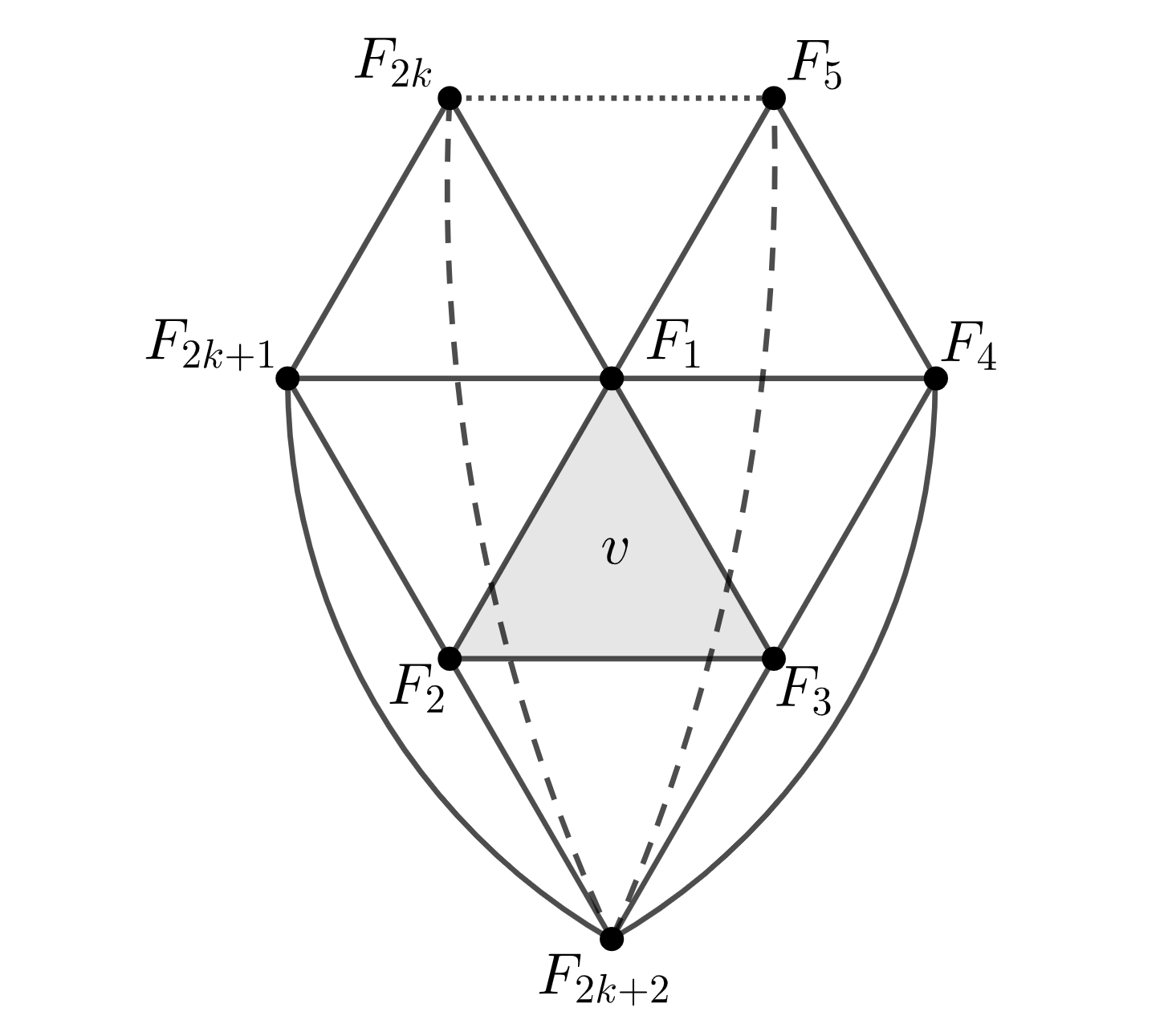}
\end{minipage}
\centering
\caption{$L_{2k}$ and its dual with label}
\label{L2k and its dual with label}
\end{figure}

\begin{lemma} \label{coefficient prism}
Using notations similar to those in Theorem \ref{dimP=3 polytope}, 
\[
c_{k+2,k+3}=\Delta_{k+2,k+3} \sum_{i=2}^{2k+1} (l_{i} \rho_{i}+\Delta_{i,i+1} \rho_{i,i+1}). 
\]
Moreover, for $4 \leq i \leq k+2$, 
\begin{equation*}
c_{i,2k+2}= -\Delta_{i-1,i}\Delta_{i-1,2k+2} \rho_{i}-\la_{1,i} \rho_{2k+2}+\rho_{i,2k+2}+\Delta_{i,2k+2} \sum_{s=4}^{i-1} (l_{s} \rho_{s}+\Delta_{s,s+1} \rho_{s,s+1}). 
\end{equation*}
And for $k+3 \leq i \leq 2k+1$, 
\begin{equation*}
c_{i,2k+2}= -\Delta_{i+1,i}\Delta_{i+1,2k+2} \rho_{i}-\la_{1,i} \rho_{2k+2}+\rho_{i,2k+2}-\Delta_{i,2k+2} \sum_{s=i+1}^{2k+1} (l_{s} \rho_{s}+\Delta_{s-1,s} \rho_{s-1,s}). 
\end{equation*}
Here all subscripts containing $i$ are taken modulo $2k$. 
\end{lemma}

\begin{proof}
By cohomology ring structure of $M(L_{2k},\Lambda)$, 
\[
\left\{
\begin{aligned}
& \la_{2,i-1} v_{i-1}v_{i}+\la_{2,i} v_{i}^{2}+\la_{2,i+1} v_{i}v_{i+1}+\la_{2,2k+2} v_{i}v_{2k+2}=0 & \qquad 5 \leq i \leq 2k; \\
& \la_{3,i-1} v_{i-1}v_{i}+\la_{3,i} v_{i}^{2}+\la_{3,i+1} v_{i}v_{i+1}+\la_{3,2k+2} v_{i}v_{2k+2}=0 & 5 \leq i \leq 2k; \\
& v_{4}^{2}+\la_{2,5} v_{4}v_{5}+\la_{2,2k+2} v_{4}v_{2k+2}=0; & \\
& \la_{3,2k} v_{2k}v_{2k+1}+v_{2k+1}^{2}+\la_{3,2k+2} v_{2k+1}v_{2k+2}=0. & \\
\end{aligned}
\right. 
\]
If $5 \leq i \leq k+2$, then elimination of $v_{i-1}v_{i}, v_{i}^{2}$ can be expressed as 
\[
\left\{
\begin{aligned}
& \Delta_{i-1,i} v_{i-1}v_{i}=\Delta_{i,i+1} v_{i}v_{i+1}+\Delta_{i,2k+2} v_{i}v_{2k+2}; \\
& v_{i}^{2}= l_{i} \cdot \Delta_{i,i+1} v_{i}v_{i+1}-\Delta_{i-1,i}\Delta_{i-1,2k+2} v_{i}v_{2k+2}. 
\end{aligned}
\right. 
\]
If $k+3 \leq i \leq 2k$, then elimination of $v_{i}^{2}, v_{i}v_{i+1}$ can be expressed as 
\[
\left\{
\begin{aligned}
& \Delta_{i,i+1} v_{i}v_{i+1}=\Delta_{i-1,i} v_{i-1}v_{i}-\Delta_{i,2k+2} v_{i}v_{2k+2}; \\
& v_{i}^{2}= l_{i} \cdot \Delta_{i-1,i} v_{i-1}v_{i}+\Delta_{i,i+1}\Delta_{i+1,2k+2} v_{i}v_{2k+2}. 
\end{aligned}
\right. 
\]
Since coefficients related to terms $v_{i}v_{i+1}$ $(4 \leq i \leq 2k)$ are exactly the same as their counterparts appear in the proof of Theorem \ref{dimP=3 polytope}, the formula for $c_{k+2,k+3}$ is nothing but a parallel version of (\ref{cyclic coefficient 2}). Meanwhile, formulas for $c_{i,2k+2}$ can be deduced from repeated use of eliminations above. 
\end{proof}

\begin{proposition}
A quasitoric manifold $M(L_{2k},\Lambda)$ is string if and only if 
\[
\la_{1,i}+\la_{2,i}+\la_{3,i} \equiv 1 \pmod{2} \qquad 4 \leq i \leq 2k+2
\]
and all 3 types of coefficients listed in Lemma \ref{coefficient prism} vanish. 
\end{proposition}

\begin{example} \label{not bundle type example}
The quasitoric manifold $M=M(L_{6},\Lambda)$ is string if $\Lambda$ is equivalent to 
\[
\left(
\begin{array}{ccc|cccc|c}
1 & 0 & 0 & 1 & 0 & 0 & 0 & 1 \\
0 & 1 & 0 & 1 & 0 & 1 & 0 & 0 \\
0 & 0 & 1 & 1 & 1 & 0 & 1 & 2 
\end{array}
\right).
\]
\end{example}

Note that the polytope $L_{6}$ is a Cartesian product of $C_{2}(6)$ and $I$. But the manifold $M$ is neither a Cartesian product, nor a quasitoric manifold of bundle type up to weakly equivariant homeomorphism (see Remark \ref{Cartesian product} for the general case). As a matter of fact, $M$ is not homeomorphic to any bundle type quasitoric manifold. 

\begin{proposition} \label{not bundle type proposition}
For any bundle type quasitoric manifold $M'$, $H^{*}(M) \not\cong H^{*}(M')$. 
\end{proposition}

\begin{proof}
Let $\alpha_{1}, \alpha_{2}, \alpha_{3}, \alpha_{4}$ and $X$ denote generators of $H^{2}(M)$ corresponding to $F_{4}, F_{5}, F_{6}, F_{7}$ and $F_{8}$ respectively. By Proposition \ref{cohomology ring}, relations in $H^{4}(M)$ are 
\[
\left\{
\begin{aligned}
& \alpha_{1}\alpha_{3}=\alpha_{1}\alpha_{4}=\alpha_{2}\alpha_{4}=0; \\
& (\alpha_{1}+X) X=0; \\
& \alpha_{1}^{2}=(\alpha_{1}+\alpha_{3}) \alpha_{2}=\alpha_{3}^{2}=0; \\
& (\alpha_{1}+\alpha_{2}+2X) \alpha_{2}=(\alpha_{2}+\alpha_{4}+2X) \alpha_{3}=(\alpha_{4}+2X) \alpha_{4}=0. 
\end{aligned}
\right.
\]

Suppose there exists a bundle type quasitoric manifold $M'$ such that $H^{*}(M) \cong H^{*}(M')$, then the orbit polytope of $M'$ is also $L_{6}$ \cite{CPS10}. Let $\Lambda'=[\ \mathrm{I}_{3}\ |\ \Lambda'_{*}\ ]$ be the characteristic matrix of $M'$ and $\varphi: H^{*}(M) \rightarrow H^{*}(M')$, $\varphi': H^{*}(M') \rightarrow H^{*}(M)$ be cohomology ring isomorphisms. Moreover, let $\beta_{1}, \beta_{2}, \beta_{3}, \beta_{4}$ and $Y$ denote generators of $H^{2}(M')$ corresponding to $F_{4}, F_{5}, F_{6}, F_{7}$ and $F_{8}$ respectively. Then one can suppose that $\varphi(\alpha_{i})=\sum_{j=1}^{4} s_{i,j} \beta_{j}+t_{i} Y$ $(1 \leq i \leq 4)$, $\varphi(X)=\sum_{j=1}^{4} r_{j} \beta_{j}+r Y$ and $\varphi'(\beta_{i})=\sum_{j=1}^{4} s'_{i,j} \alpha_{j}+t'_{i} X$ $(1 \leq i \leq 4)$, $\varphi'(Y)=\sum_{j=1}^{4} r'_{j} \alpha_{j}+r' X$. In this way, 
\begin{equation} \label{cohomology ring isomorphism}
\mathrm{det}
\begin{pmatrix}
s_{1,1} & s_{1,2} & s_{1,3} & s_{1,4} & t_{1} \\
s_{2,1} & s_{2,2} & s_{2,3} & s_{2,4} & t_{2} \\
s_{3,1} & s_{3,2} & s_{3,3} & s_{3,4} & t_{3} \\
s_{4,1} & s_{4,2} & s_{4,3} & s_{4,4} & t_{4} \\
r_{1} & r_{2} & r_{3} & r_{4} & r 
\end{pmatrix}
= \pm 1 
\qquad 
\mathrm{det}
\begin{pmatrix}
s'_{1,1} & s'_{1,2} & s'_{1,3} & s'_{1,4} & t'_{1} \\
s'_{2,1} & s'_{2,2} & s'_{2,3} & s'_{2,4} & t'_{2} \\
s'_{3,1} & s'_{3,2} & s'_{3,3} & s'_{3,4} & t'_{3} \\
s'_{4,1} & s'_{4,2} & s'_{4,3} & s'_{4,4} & t'_{4} \\
r'_{1} & r'_{2} & r'_{3} & r'_{4} & r' 
\end{pmatrix}
= \pm 1. 
\end{equation}

\noindent \emph{Case 1.} $M'$ is a $\C P^{1}$-bundle over 4-dimensional quasitoric manifold, i.e., 
\[
\Lambda'_{*}=
\left(
\begin{array}{cccc|c}
x & y & z & w & 1 \\
\hline
\multicolumn{4}{c|}{\multirow{2}*{$*_{2 \times 4}$}} & 0 \\
& & & & 0
\end{array}
\right) 
\]
with $x, y, z, w \in \Z$. Note that among relations in $H^{4}(M')$, $\beta_{j} Y$ $(1 \leq j \leq 4)$ and $Y^{2}$ only appear in $x \beta_{1} Y+y \beta_{2} Y+z \beta_{3} Y+w \beta_{4} Y+Y^{2}=0$. Therefore, $\varphi(\alpha_{1}) \varphi(\alpha_{3})=(\sum_{j=1}^{4} s_{1,j} \beta_{j}+t_{1} Y)(\sum_{j=1}^{4} s_{3,j} \beta_{j}+t_{3} Y)=0$ induces the following equations: 
\begin{equation} \label{coefficient relation case 1}
\left\{
\begin{aligned}
& s_{1,1} t_{3}+s_{3,1} t_{1}=x t_{1} t_{3}; \\
& s_{1,2} t_{3}+s_{3,2} t_{1}=y t_{1} t_{3}; \\
& s_{1,3} t_{3}+s_{3,3} t_{1}=z t_{1} t_{3}; \\
& s_{1,4} t_{3}+s_{3,4} t_{1}=w t_{1} t_{3}. 
\end{aligned}
\right. 
\end{equation}
By (\ref{cohomology ring isomorphism}) and (\ref{coefficient relation case 1}), $t_{1}=0$ is equivalent to $t_{3}=0$ and $t_{1} \equiv 0 \pmod{2}$ is equivalent to $t_{3} \equiv 0 \pmod{2}$. Parallel results can be deduced from the vanishing of $\varphi(\alpha_{1}) \varphi(\alpha_{4})$ and $\varphi(\alpha_{2}) \varphi(\alpha_{4})$. Thus, $\{t_{i}\}_{i=1}^{4}$ have the same parity and either all of them or none of them vanishes. 

On the other hand, by elementary transformation based on (\ref{coefficient relation case 1}), 
\[
t_{1} t_{3} \cdot \mathrm{det}
\begin{pmatrix}
x & y & z & w & 2 \\
s_{2,1} & s_{2,2} & s_{2,3} & s_{2,4} & t_{2} \\
s_{3,1} & s_{3,2} & s_{3,3} & s_{3,4} & t_{3} \\
s_{4,1} & s_{4,2} & s_{4,3} & s_{4,4} & t_{4} \\
r_{1} & r_{2} & r_{3} & r_{4} & r 
\end{pmatrix}
=
t_{3} \cdot \mathrm{det}
\begin{pmatrix}
s_{1,1} & s_{1,2} & s_{1,3} & s_{1,4} & t_{1} \\
s_{2,1} & s_{2,2} & s_{2,3} & s_{2,4} & t_{2} \\
s_{3,1} & s_{3,2} & s_{3,3} & s_{3,4} & t_{3} \\
s_{4,1} & s_{4,2} & s_{4,3} & s_{4,4} & t_{4} \\
r_{1} & r_{2} & r_{3} & r_{4} & r 
\end{pmatrix}
= \pm t_{3}. 
\]
If $t_{3}=0$, then $t_{i}=0$ $(1 \leq i \leq 4)$ and $r= \pm 1$. Hence, $\varphi(\alpha_{4}+2X) \varphi(\alpha_{4})=[\sum_{j=1}^{4} (s_{4,j}+2 r_{j}) \beta_{j}\pm 2Y](\sum_{j=1}^{4} s_{4,j} \beta_{j}) \neq 0$, resulting in contradiction. If $t_{3} \neq 0$, then $t_{i} \equiv 1 \pmod{2}$ $(1 \leq i \leq 4)$. Hence, (\ref{coefficient relation case 1}) along with the parallel result induced by $\varphi(\alpha_{2}) \varphi(\alpha_{4})=0$ indicates 
\[
\left\{
\begin{aligned}
& s_{1,1}+s_{3,1} \equiv s_{2,1}+s_{4,1} \equiv x \pmod{2}; \\
& s_{1,2}+s_{3,2} \equiv s_{2,2}+s_{4,2} \equiv y \pmod{2}; \\
& s_{1,3}+s_{3,3} \equiv s_{2,3}+s_{4,3} \equiv z \pmod{2}; \\
& s_{1,4}+s_{3,4} \equiv s_{2,4}+s_{4,4} \equiv w \pmod{2}. 
\end{aligned}
\right. 
\]
This leads to $\sum_{i=1}^{4} s_{i,1} \equiv \sum_{i=1}^{4} s_{i,2} \equiv \sum_{i=1}^{4} s_{i,3} \equiv \sum_{i=1}^{4} s_{i,4} \equiv \sum_{i=1}^{4} t_{i} \equiv 0 \pmod{2}$ and contradicts (\ref{cohomology ring isomorphism}). 

\vspace{0.5cm}

\noindent \emph{Case 2.} $M'$ is a 4-dimensional quasitoric manifold-bundle over $\C P^{1}$, i.e., 
\[
\Lambda'_{*}=
\left(
\begin{array}{cccc|c}
0 & 0 & 0 & 0 & 1 \\
\hline
\multicolumn{4}{c|}{*_{2 \times 4}} & *_{2 \times 1}
\end{array}
\right). 
\]
Note that among relations in $H^{4}(M)$: $\alpha_{1} X, X^{2}$ only appear in $\alpha_{1} X+X^{2}=0$; $\alpha_{2} X, \alpha_{2}^{2}$ only appear in $\alpha_{1} \alpha_{2}+\alpha_{2}^{2}+2\alpha_{2} X=0$; $\alpha_{3} X, \alpha_{3} \alpha_{4}$ only appear in $\alpha_{2} \alpha_{3}+\alpha_{3} \alpha_{4}+2\alpha_{3} X=0$ and $\alpha_{4} X, \alpha_{4}^{2}$ only appear in $\alpha_{4}^{2}+2\alpha_{4} X=0$. Therefore, $\varphi'(\beta_{1}) \varphi'(\beta_{3})=(\sum_{j=1}^{4} s'_{1,j} \alpha_{j}+t'_{1} X)(\sum_{j=1}^{4} s'_{3,j} \alpha_{j}+t'_{3} X)=0$ induces the following equations: 
\begin{equation} \label{coefficient relation case 2}
\left\{
\begin{aligned}
& s'_{1,1} t'_{3}+s'_{3,1} t'_{1}=t'_{1} t'_{3}; \\
& s'_{1,2} t'_{3}+s'_{3,2} t'_{1}=2 s'_{1,2} s'_{3,2}; \\
& s'_{1,3} t'_{3}+s'_{3,3} t'_{1}=2 (s'_{1,3} s'_{3,4}+s'_{1,4} s'_{3,3}); \\
& s'_{1,4} t'_{3}+s'_{3,4} t'_{1}=2 s'_{1,4} s'_{3,4}. 
\end{aligned}
\right. 
\end{equation}
Similar to arguments in \emph{Case 1}, (\ref{coefficient relation case 2}) and the parallel results indicate that $\{t'_{i}\}_{i=1}^{4}$ have the same parity. Moreover, $\varphi'(Y^2)=(\sum_{j=1}^{4} r'_{j} \alpha_{j}+r' X)^{2}=0$ implies $2 r'_{1}r'=(r')^{2}$. In particular, $\{t'_{i}\}_{i=1}^{4}$ must be odd since $r'$ is even. This leads to $\sum_{i=1}^{4} s'_{i,1} \equiv \sum_{i=1}^{4} s'_{i,2} \equiv \sum_{i=1}^{4} s'_{i,3} \equiv \sum_{i=1}^{4} s'_{i,4} \equiv \sum_{i=1}^{4} t'_{i} \equiv 0 \pmod{2}$ and contradicts (\ref{cohomology ring isomorphism}). 
\end{proof}

Although the manifold in Example \ref{not bundle type example} is not of bundle type, it is indeed weakly equivariantly homeomorphic to the equivariant edge connected sum of two bundle type quasitoric manifolds. 

\begin{definition} \label{equivariant edge connected sum definition} (see \cite{LY11} for real case based on $\Z_{2}$-coloring) Given $2n$-dimensional quasitoric manifolds $M(P_{1},\Lambda_{1})$ and $M(P_{2},\Lambda_{2})$, Let $\mathcal{F}(P_{1})=\{F_{i}\}_{i=1}^{m_{1}}$, $\mathcal{F}(P_{2})=\{F'_{j}\}_{j=1}^{m_{2}}$ denote facet sets and $\Lambda_{1}=(\boldsymbol{\la_{1}}, \cdots, \boldsymbol{\la_{m_{1}}})$, $\Lambda_{2}=(\boldsymbol{\la'_{1}}, \cdots, \boldsymbol{\la'_{m_{2}}})$ denote characteristic matrices. Suppose there are edges $a= \cap_{k=1}^{n-1} F_{i_{k}}, a'=\ \cap_{k=1}^{n-1} F'_{j_{k}}$ and vertices $v=a \cap F_{i_{n}}, w=a \cap F_{i_{n+1}}, v'=a' \cap F'_{j_{n}}, w'=a' \cap F'_{j_{n+1}}$ satisfying $\boldsymbol{\la_{i_{k}}}=\boldsymbol{\la'_{j_{k}}}$ for $1 \leq k \leq n+1$. Then reorder facets such that $F_{k}=F_{i_{k}}, F'_{k}=F'_{j_{k}}$ for $1 \leq k \leq n+1$. In this way, the \emph{equivariant edge connected sum} at $a, a'$ is defined to be a quasitoric manifold $M(P,\Lambda)$, where $P=P_{1} \#^{e} P_{2}$ is the edge connected sum of $P_{1}$ and $P_{2}$ at $a, a'$ (see Figure \ref{cube edge connected sum} as an example) and 
\[
\Lambda=(\boldsymbol{\la_{1}}, \cdots, \boldsymbol{\la_{m_{1}}}, \boldsymbol{\la'_{n+2}}, \cdots, \boldsymbol{\la'_{m_{2}}}). 
\]
It should be pointed out that $\Lambda$ is NOT in refined form here. The explicit notation should be $M(P_{1},\Lambda_{1}) \widetilde{\#^{e}}_{a, a'} M(P_{2},\Lambda_{2})$, but it can be simplified as $M(P_{1},\Lambda_{1}) \widetilde{\#^{e}} M(P_{2},\Lambda_{2})$ when there is no confusion. 
\end{definition}

\begin{figure}[htb]
\centering
\includegraphics[height=4cm]{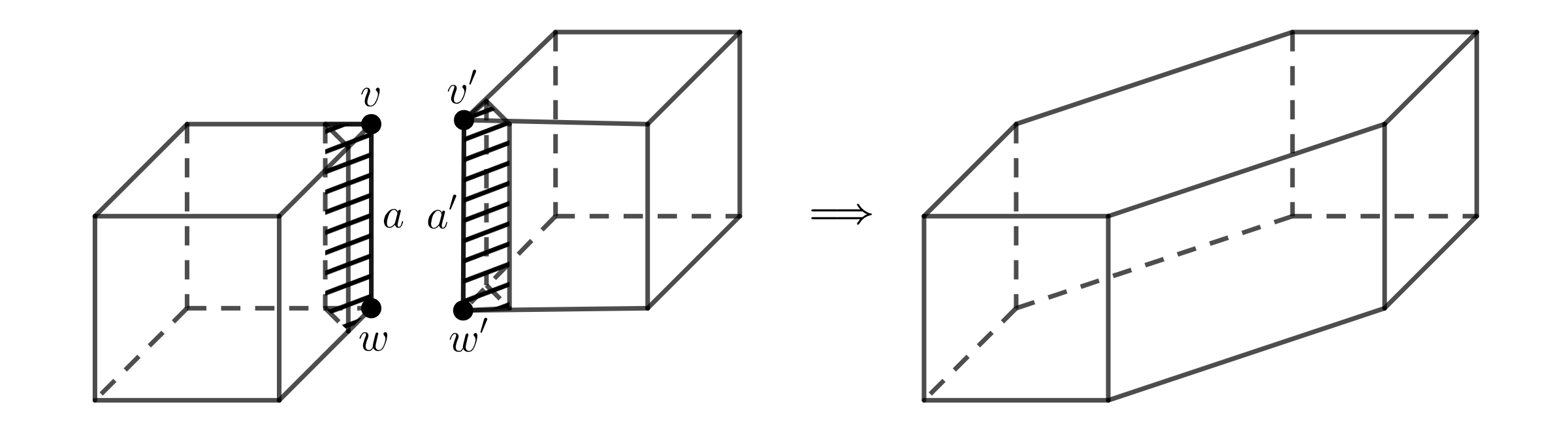}
\caption{Edge connected sum of two cubes}
\label{cube edge connected sum}
\end{figure}

Let $(\boldsymbol{\la_{1}}, \cdots, \boldsymbol{\la_{8}})$ denote the characteristic matrix in Example \ref{not bundle type example}. Note that $L_{6}=L_{4} \#^{e} L_{4}$ and $\mathrm{det}(\boldsymbol{\la_{1}}, \boldsymbol{\la_{4}}, \boldsymbol{\la_{7}})=\mathrm{det}(\boldsymbol{\la_{8}}, \boldsymbol{\la_{4}}, \boldsymbol{\la_{7}})=1$. Thus, $M(L_{6},\Lambda)=M(L_{4},\Lambda_{1}) \widetilde{\#^{e}} M(L_{4},\Lambda_{2})$, where 
\[
\Lambda_{1}=
\begin{pmatrix}
1 & 0 & 0 & 1 & 0 & 1 \\
0 & 1 & 0 & 1 & 0 & 0 \\
0 & 0 & 1 & 1 & 1 & 2 
\end{pmatrix} 
\qquad 
\Lambda_{2}=
\begin{pmatrix}
1 & 1 & 0 & 0 & 0 & 1 \\
0 & 1 & 0 & 1 & 0 & 0 \\
0 & 1 & 1 & 0 & 1 & 2 
\end{pmatrix}. 
\]
Since $\Lambda_{1}$ and $\Lambda_{2}$ are equivalent to 
\[
\begin{pmatrix}
1 & 0 & 0 & 2 & 1 & 1 \\
0 & 1 & 0 & 1 & 1 & 0 \\
0 & 0 & 1 & 0 & 1 & 0 
\end{pmatrix}, 
\]
both $M(L_{4},\Lambda_{1})$ and $M(L_{4},\Lambda_{2})$ are of bundle type ($\C P^{1}$-bundle over $\C P^{2} \# \overline{\C P^{2}}$). Moreover, direct computation yields that $M(L_{4},\Lambda_{1})$ and $M(L_{4},\Lambda_{2})$ are string. 

In general, all string quasitoric manifolds over prisms can be constructed from bundle type string quasitoric manifolds over prisms via equivariant edge connected sum. 

\begin{theorem} \label{equivariant edge connected sum theorem}
If a quasitoric manifold $M(L_{2k},\Lambda)$ is string, then there exist quasitoric manifolds $\{M(L_{2k_{i}},\Lambda_{i})\}_{i=1}^{s}$ such that \\
(1) $M(L_{2k_{i}},\Lambda_{i})$ is of bundle type and string for $1 \leq i \leq s$; \\
(2) $M(L_{2k},\Lambda)=M(L_{2k_{1}},\Lambda_{1}) \widetilde{\#^{e}} \cdots \widetilde{\#^{e}} M(L_{2k_{s}},\Lambda_{s})$ up to weakly equivariant homeomorphism. 
\end{theorem}

\begin{proof}
Let the basis of $H^{4}(M(L_{2k},\Lambda))$ be $\{v_{i}v_{2k+2}\}_{4 \leq i \leq 2k+1}$ along with $v_{k+2}v_{k+3}$. The vanishing of $p_{1}(M(L_{2k},\Lambda))$ requires that $c_{4,2k+2}=c_{2k+1,2k+2}=0$, i.e., 
\begin{equation} \label{coefficient prism 1}
\left\{
\begin{aligned}
& \la_{2,2k+2} \rho_{4}+\la_{1,4} \rho_{2k+2}=\rho_{4,2k+2}; \\
& \la_{3,2k+2} \rho_{2k+1}+\la_{1,2k+1} \rho_{2k+2}=\rho_{2k+1,2k+2}. 
\end{aligned}
\right.
\end{equation}
On the other hand, (\ref{nonsingular}) requires $|1-\la_{1,4}\la_{2,2k+2}|=|1-\la_{1,2k+1}\la_{3,2k+2}|=1$. Note that $\la_{1,4}\la_{2,2k+2}=2$ leads to the following contradiction against (\ref{coefficient prism 1}): 
\begin{align*}
|\la_{2,2k+2} \rho_{4}+\la_{1,4} \rho_{2k+2}| \geq & \sum_{i=1}^{3} (\la_{i,4}^{2}+\la_{i,2k+2}^{2})+3 \\
\geq & 2 \sum_{i=1}^{3} |\la_{i,4}\la_{i,2k+2}|+3 \\
> & |\rho_{4,2k+2}|. 
\end{align*}
Thus, $\la_{1,4}\la_{2,2k+2}=\la_{1,2k+1}\la_{3,2k+2}=0$. 

\vspace{0.5cm}

\noindent \emph{Case 1: $\it{\la_{2,2k+2}=\la_{3,2k+2}=0}$.} $M(L_{2k},\Lambda)$ itself is of bundle type. 

\vspace{0.5cm}

\noindent \emph{Case 2: $\it{\la_{2,2k+2}=0; \la_{3,2k+2} \neq 0}$ ($\it{\la_{2,2k+2} \neq 0; \la_{3,2k+2}=0}$ is a parallel case).} Then $\la_{1,2k+1}=0$ and (\ref{coefficient prism 1}) implies 
\[
\la_{1,4}\la_{3,2k+2}-2 \la_{3,4}=\la_{2,2k+1}=0. 
\]

In this way, $\mathrm{det}(\boldsymbol{\la_{1}}, \boldsymbol{\la_{4}}, \boldsymbol{\la_{2k+1}})=\mathrm{det}(\boldsymbol{\la_{2k+2}}, \boldsymbol{\la_{4}}, \boldsymbol{\la_{2k+1}})=1$. Hence, for $k \geq 3$, $M(L_{2k},\Lambda)$ can be decomposed as $M(L_{4},\Lambda_{1}) \widetilde{\#^{e}} M(L_{2k-2},\Lambda_{2})$, where 
\[
\Lambda_{1}=
\begin{pmatrix}
1 & 0 & 0 & \la_{1,4} & 0 & 1 \\
0 & 1 & 0 & 1 & 0 & 0 \\
0 & 0 & 1 & \la_{3,4} & 1 & \la_{3,2k+2} 
\end{pmatrix} 
\qquad 
\Lambda_{2}=
\begin{pmatrix}
1 & \la_{1,4} & \la_{1,5} & \cdots & 0 & 1 \\
0 & 1 & \la_{2,5} & \cdots & 0 & 0 \\
0 & \la_{3,4} & \la_{3,5} & \cdots & 1 & \la_{3,2k+2} 
\end{pmatrix}. 
\]
Since $\Lambda_{1}$ is equivalent to 
\[
\begin{pmatrix}
1 & 0 & 0 & \la_{3,2k+2} & \la_{3,4} & 1 \\
0 & 1 & 0 & 1 & \la_{1,4} & 0 \\
0 & 0 & 1 & 0 & 1 & 0 
\end{pmatrix}, 
\]
$M(L_{4},\Lambda_{1})$ is of bundle type. And direct verification leads to the vanishing of $p_{1}(M(L_{4},\Lambda_{1}))$. In particular, if $k=2$, then $M(L_{4},\Lambda)$ itself is of bundle type and string. Meanwhile, 
\[
p_{1}(M(L_{2k-2},\Lambda_{2}))=\sum_{i=5}^{2k} v_{i}^{2}+(\sum_{i=4}^{2k} \la_{1,i} v_{i})^{2}+(\sum_{i=5}^{2k} \la_{2,i} v_{i})^{2}+(\sum_{i=4}^{2k} \la_{3,i} v_{i}+\la_{3,2k+2} v_{2k+2})^{2}. 
\]
It should be noted that $p_{1}(M(L_{2k},\Lambda))$ has exactly the same expression. Besides, for $M(L_{2k},\Lambda)$ and $M(L_{2k-2},\Lambda_{2})$, relations among $\{v_{i}v_{j}\}_{4 \leq i, j \leq 2k+1}$ and $\{v_{i}v_{2k+2}\}_{i=4}^{2k+2}$ are identical, except for $v_{4}v_{2k+1}$. Since $v_{4}v_{2k+1}$ does not appear in the expression above, $p_{1}(M(L_{2k},\Lambda))=0$ induces $p_{1}(M(L_{2k-2},\Lambda_{2}))=0$. 

\vspace{0.5cm}

\noindent \emph{Case 3: $\it{\la_{2,2k+2} \neq 0; \la_{3,2k+2} \neq 0}$.} Then $\la_{1,4}=\la_{1,2k+1}=0$ and (\ref{coefficient prism 1}) implies 
\[
\left\{
\begin{aligned}
& \la_{2,2k+2}\la_{3,4}^{2}=2 \la_{3,4}\la_{3,2k+2}; \\
& \la_{3,2k+2}\la_{2,2k+1}^{2}=2 \la_{2,2k+1}\la_{2,2k+2}. 
\end{aligned}
\right.
\]
If $\la_{3,4}\la_{2,2k+1}=0$, then arguments in \emph{Case 2} can be applied to decompose $M(L_{2k},\Lambda)$ via equivariant edge connected sum. Otherwise, $\la_{3,4}\la_{2,2k+1}=4$. Along with restrictions given by spin property, $\la_{3,4}=\la_{2,2k+1}=2$ up to equivalence, and then $\la_{2,2k+2}=\la_{3,2k+2}=a$. 

Suppose there exists $5 \leq t \leq k+2$ such that $\la_{1,4}=\cdots=\la_{1,t-1}=0$ but $\la_{1,t} \neq 0$. With necessary column sign permutations, one can assume $\Delta_{i-1,i}=-1$ for $4 \leq i \leq t$. Then determinant restrictions indicate 
\[
|1-\la_{1,t}a (\la_{3,t-1}-\la_{2,t-1})|=1. 
\]
Since $\la_{1,t} \neq 0, a \neq 0$ by assumption and $\la_{3,t-1}-\la_{2,t-1} \equiv 1 \pmod{2}$ by spin property, $\la_{3,t-1}-\la_{2,t-1}=\frac{2}{\la_{1,t}a}=\pm 1$. Now check the vanishing of coefficients $c_{t-1,2k+2}$ and $c_{t,2k+2}$ with formulas in Lemma \ref{coefficient prism}: $c_{t-1,2k+2}=0$ indicates 
\begin{align*}
\sum_{s=4}^{t-2} (l_{s} \rho_{s}+\Delta_{s,s+1} \rho_{s,s+1}) 
= &\ \frac{-1}{\Delta_{t-1,2k+2}} (\Delta_{t-2,2k+2} \rho_{t-1}+\rho_{t-1,2k+2}) \\
= &\ (\la_{3,t-1}-\la_{2,t-1})[(\la_{2,t-2}-\la_{3,t-2})\rho_{t-1} \\
& +2(\la_{2,t-1}+\la_{3,t-1})]. 
\end{align*}
Substitution into $c_{t,2k+2}=0$ yields 
\begin{equation} \label{c_{t-1,2k+2} and c_{t,2k+2}}
\begin{aligned}
4+\rho_{t}= 
&\ (\la_{3,t-1}-\la_{2,t-1})[2(\la_{2,t}+\la_{3,t})-(\la_{2,t}-\la_{3,t}) \rho_{t-1,t}] \\
& +(\la_{2,t}-\la_{3,t})[2(\la_{2,t-1}+\la_{3,t-1})+(\la_{2,t-2}-\la_{3,t-2}) \rho_{t-1}] \\
& +(\la_{3,t-1}-\la_{2,t-1})(\la_{2,t}-\la_{3,t})(\la_{2,t}\la_{3,t-2}-\la_{2,t-2}\la_{3,t}) \rho_{t-1}. 
\end{aligned}
\end{equation}
When $\la_{3,t-1}-\la_{2,t-1}=1$, submatrix $(\la_{i,j})_{1 \leq i \leq 3, t-2 \leq j \leq t}$ can be expressed as 
\[
\begin{pmatrix}
0 & 0 & \la_{1,t} \\
bc-1 & b & bd+1 \\
bc+c-1 & b+1 & bd+d+1
\end{pmatrix}
\]
with $b, c, d \in \Z$. And (\ref{c_{t-1,2k+2} and c_{t,2k+2}}) is simplified to $\la_{1,t}^{2}+d^{2}+3=0$, resulting in contradiction. When $\la_{3,t-1}-\la_{2,t-1}=-1$, submatrix $(\la_{i,j})_{1 \leq i \leq 3, t-2 \leq j \leq t}$ can be expressed as 
\[
\begin{pmatrix}
0 & 0 & \la_{1,t} \\
bc+c+1 & b+1 & bd+d-1 \\
bc+1 & b & bd-1 
\end{pmatrix}
\]
with $b, c, d \in \Z$. And (\ref{c_{t-1,2k+2} and c_{t,2k+2}}) is simplified to $\la_{1,t}^{2}+d^{2}(2b+1)^{2}+3=0$, resulting in contradiction again. Therefore, $\la_{1,4}=\cdots=\la_{1,k+2}=0$. Likewise, $\la_{1,2k+1}=\cdots=\la_{1,k+3}=0$. And $M(L_{2k},\Lambda)$ itself is of bundle type. 

In conclusion, $M(L_{2k},\Lambda)$ is either of bundle type, or an equivariant edge connected sum $M(L_{4},\Lambda_{1}) \widetilde{\#^{e}} M(L_{2k-2},\Lambda_{2})$, where $M(L_{4},\Lambda_{1})$ is of bundle type and string, and $M(L_{2k-2},\Lambda_{2})$ is string. The proof finishes with induction. 
\end{proof} 

\begin{remark}
Combine the proof with results in Section \ref{dimP=2}, $M(L_{2k_{i}},\Lambda_{i})$ $(1 \leq i \leq s-1)$ can be assumed as the total space of a 3-stage $\C P^{1}$-bundle tower: 
\[
B^{6} \xrightarrow{\C P^{1}} B^{4} \xrightarrow{\C P^{1}} B^{2} \xrightarrow{\C P^{1}} {pt}.
\]
And $M(L_{2k_{s}},\Lambda_{s})$ is either a $\C P^{1}$-bundle over 4-dimensional quasitoric manifold, or a $\widetilde{\#}_{k_{s}-1} (\C P^{1} \times \C P^{1})$-bundle over $\C P^{1}$. In addition, with discussions similar to those in \emph{Case 2}, one can deduce $[M(L_{2k},\Lambda)]=[M(L_{2k_{s}},\Lambda_{s})]=0$ in $\Omega_{6}^{O}$, $\Omega_{6}^{SO}$ and there exists an omniorientation such that $[M(L_{2k},\Lambda)]=0$ in $\Omega_{6}^{U}$. 
\end{remark}

\subsection{$\mathrm{dim}P=4$} \label{dimP=4}

General results on 8-dimensional string quasitoric manifolds and their orbit polytopes are much more difficult to reach. On one hand, there are 4-dimensional simple polytopes that can not be realized as the orbit polytope of any quasitoric manifold, such as dual of cyclic polytopes with more than 7 facets \cite{Ha15}. On the other hand, the property parallel to Proposition \ref{dimP=2 polytope} and \ref{dimP=3 polytope} is NOT valid, as illustrated in the example below. 

\begin{example} \label{5-colorable}
The quasitoric manifold $M(C_{2}(4) \times C_{2}(5),\Lambda)$ is string if $\Lambda$ is equivalent to 
\[
\left(
\begin{array}{cc|cc|cc|ccc}
1 & 0 & 0 & 0 & 1 & 0 & 0 & 1 & 0 \\
0 & 1 & 0 & 0 & 0 & 1 & 2 & 2 & 2 \\
\hline
0 & 0 & 1 & 0 & 0 & 0 & 1 & 1 & 0 \\
0 & 0 & 0 & 1 & 0 & 0 & 0 & 1 & 1 
\end{array}
\right). 
\]
Here facets are labeled such that $F_{1}, F_{2}, F_{5}, F_{6}$ (resp. $F_{3}, F_{4}, F_{7}, F_{8}, F_{9}$) correspond to facets of $C_{2}(4)$ (resp. $C_{2}(5)$). This example can be generalized to obtain string quasitoric manifolds over $C_{2}(2s) \times C_{2}(2t+1) \times I^{n}$ for all $s, t \geq 2$ and $n \geq 0$. 
\end{example}

When restricted to the product of two polygons, we are led to the following characterization similar to Proposition \ref{dimP=2 polytope}: 

\begin{proposition} \label{dimP=4 polytope}
$P=C_{2}(m_{1}) \times C_{2}(m_{2})$ can be realized as the orbit polytope of a string quasitoric manifold if and only if $m_{1}, m_{2} \geq 4$ and $m_{1}m_{2} \equiv 0 \pmod{2}$. 
\end{proposition}

\begin{proof}
With pull back of the linear model and Example \ref{5-colorable}, it remains to prove the necessity. 

Label the facets of $P$ such that $\{F_{i}\}_{i=1}^{m_{1}}$ correspond to facets of $C_{2}(m_{1})$ and $\{F_{i}\}_{i=m_{1}+1}^{m_{1}+m_{2}}$ correspond to facets of $C_{2}(m_{2})$. Fix the initial vertex $w=F_{1} \cap F_{2} \cap F_{m_{1}+1} \cap F_{m_{1}+2}$. For any quasitoric manifold $M(P,\Lambda)$, let the basis of $H^{2}(M(P,\Lambda))$ be $\{v_{i}\}_{i=3}^{m_{1}}$ and $\{v_{i}\}_{i=m_{1}+3}^{m_{1}+m_{2}}$. 

Without loss of generality, suppose $m_{1} \leq m_{2}$. If $m_{1}=3$, then non-vanishing element $v_{3}^{2}$ does not appear in any relation in $H^{4}(M(P,\Lambda))$ (see Key Observation in Section \ref{few facets case} for the general case). Thus, $p_{1}(M(P,\Lambda)) \neq 0$, leading to contradiction. 

Now suppose $m_{1}=2s+1, m_{2}=2t+1$ with $2 \leq s \leq t$. Let $\mu_{i,j}$ characterizes the parity of $\la_{i,j}$, i.e., $\mu_{i,j} \equiv \la_{i,j} \pmod{2}$. For $1 \leq i \leq 2s+1$, set 
\[
\left\{
\begin{aligned}
& \mathcal{O}_{1}=\{ i\ |\ \mu_{1,i}=\mu_{2,i}, \mu_{j,i-1} \neq \mu_{j,i+1} \Leftrightarrow j=1, 2, 3, 4 \}; \\
& \mathcal{O}_{2}=\{ i\ |\ \mu_{1,i}=\mu_{2,i}, \mu_{j,i-1} \neq \mu_{j,i+1} \Leftrightarrow j=1, 2 \}; \\
& \mathcal{O}_{3}=\{ i\ |\ \mu_{1,i}=\mu_{2,i}, \mu_{j,i-1} \neq \mu_{j,i+1} \Leftrightarrow j=3, 4 \}; \\
& \mathcal{E}=\{ i\ |\ \mu_{1,i} \neq \mu_{2,i}, \mu_{j,i} \neq \mu_{j,i+1} \Leftrightarrow j=1, 2, 3, 4 \}; 
\end{aligned}
\right. 
\]
with subscripts taken modulo $2s+1$. Similarly, for $2s+2 \leq i \leq 2s+2t+2$, set 
\[
\left\{
\begin{aligned}
& \mathcal{O}_{1}'=\{ i\ |\ \mu_{3,i}=\mu_{4,i}, \mu_{j,i-1} \neq \mu_{j,i+1} \Leftrightarrow j=1, 2, 3, 4 \}; \\
& \mathcal{O}_{2}'=\{ i\ |\ \mu_{3,i}=\mu_{4,i}, \mu_{j,i-1} \neq \mu_{j,i+1} \Leftrightarrow j=3, 4 \}; \\
& \mathcal{O}_{3}'=\{ i\ |\ \mu_{3,i}=\mu_{4,i}, \mu_{j,i-1} \neq \mu_{j,i+1} \Leftrightarrow j=1, 2 \}; \\
& \mathcal{E}'=\{ i\ |\ \mu_{3,i} \neq \mu_{4,i}, \mu_{j,i} \neq \mu_{j,i+1} \Leftrightarrow j=1, 2, 3, 4 \}; 
\end{aligned}
\right. 
\]
with subscripts taken modulo $2t+1$. Direct check on restrictions imposed by (\ref{nonsingular}) along with spin property yields 
\begin{enumerate}[(1)]
\item $\mathcal{O}_{1} \neq \emptyset \Rightarrow \mathcal{O}_{2}'=\mathcal{O}_{3}'=\emptyset$ and $\mathcal{O}_{1}' \neq \emptyset \Rightarrow \mathcal{O}_{2}=\mathcal{O}_{3}=\emptyset$; 
\item $\mathcal{O}_{2} \neq \emptyset \Rightarrow \mathcal{O}_{1}'=\mathcal{O}_{2}'=\emptyset$ and $\mathcal{O}_{2}' \neq \emptyset \Rightarrow \mathcal{O}_{1}=\mathcal{O}_{2}=\emptyset$; 
\item $\mathcal{E} \neq \emptyset \Rightarrow \mathcal{E}'=\emptyset$ and $\mathcal{E}' \neq \emptyset \Rightarrow \mathcal{E}=\emptyset$. 
\end{enumerate}
On the other hand, $\la_{1,1}=\la_{2,2}=1, \la_{1,2}=\la_{2,1}=0$ and $\la_{3,2s+2}=\la_{4,2s+3}=1, \la_{3,2s+3}=\la_{4,2s+2}=0$ require that cardinal numbers $|\mathcal{O}_{1} \sqcup \mathcal{O}_{2}|\equiv |\mathcal{O}_{1}' \sqcup \mathcal{O}_{2}'| \equiv 1 \pmod{2}$. Thus, $\mathcal{O}_{2}=\mathcal{O}_{3}=\mathcal{O}_{2}'=\mathcal{O}_{3}'=\emptyset$ and $|\mathcal{O}_{1}| \equiv |\mathcal{O}_{1}'| \equiv 1 \pmod{2}$ by (1) and (2). Consequently, $\la_{3,1}=\la_{3,2}=\la_{4,1}=\la_{4,2}=0$ and $\la_{1,2s+2}=\la_{1,2s+3}=\la_{2,2s+2}=\la_{2,2s+3}=0$ require that $|\mathcal{E}| \equiv |\mathcal{E}'| \equiv 1 \pmod{2}$. But this is a contradiction against (3). 
\end{proof}

\section{Few facets case} \label{few facets case}
\setcounter{equation}{0}

\begin{observation} \label{key observation}
For an $n$-dimensional simple polytope $P$ with facet set $\mathcal{F}(P)$, if there exist facets $F$ and $\{F_{j_{i}}\}_{i=1}^{n}$ such that $\bigcap_{i=1}^{n} F_{j_{i}} \neq \emptyset$ and $F \cap F_{j_{i}} \neq \emptyset$ for $1 \leq i \leq n$, then one can relabel elements of $\mathcal{F}(P)$ with $F'_{i}=F_{j_{i}}$ $(1 \leq i \leq n)$ and $F'_{n+1}=F$. In this way, $v_{n+1}^{2}$ does NOT appear in any relation in $H^{4}(M(P,\Lambda))$. As a result, non-zero element $v_{n+1}^{2}$ must appear in the expression of $p_{1}(M(P,\Lambda))$ and its coefficient is equal to $\sum_{i=1}^{n}\la_{i(n+1)}^2+1 \neq 0$. In conclusion, $P$ can not be realized as the orbit polytope of a string quasitoric manifold. 
\end{observation}

This observation rules out a large amount of simple polytopes such as $P=\prod_{i=1}^{k} P_{i}$ with some $P_{i}$ having a 2-neighborly dual and $P$ with a triangular $2$-face. In particular, a necessary and sufficient condition is obtained in the case of product of simplices: 
\begin{proposition} \label{product of simplices}
$P=\prod_{i=1}^{k}\Delta^{n_{i}}$ can be realized as the orbit polytope of a string quasitoric manifold if and only if $n_{i}=1$ for all $1 \leq i \leq k$, i.e., $P$ is the cube $I^{k}$.
\end{proposition}
On the other hand, G. Blind and R. Blind classified all triangle-free simple polytopes with few facets: 
\begin{theorem}
\emph{\cite[Theorem 3]{BB92}} If an $n$-dimensional simple polytope $P$ is triangle-free, then the number of facets $f_{n-1}(P) \geq 2n$. Moreover, \\
(1) $f_{n-1}(P)=2n \Rightarrow P=I^{n}$; \\
(2) $f_{n-1}(P)=2n+1 \Rightarrow P=C_{2}(5) \times I^{n-2}$; \\
(3) $f_{n-1}(P)=2n+2 \Rightarrow P=C_{2}(6) \times I^{n-2}$ or $Q \times I^{n-3}$ or $C_{2}(5) \times C_{2}(5) \times I^{n-4}$ where $Q$ can be obtained from an edge cut of $C_{2}(5) \times I$ (see Figure \ref{Q3 as an edge cut}).
\end{theorem}

\begin{figure}[htb]
\begin{minipage}[t]{0.45\linewidth}
\centering
\includegraphics[height=4cm]{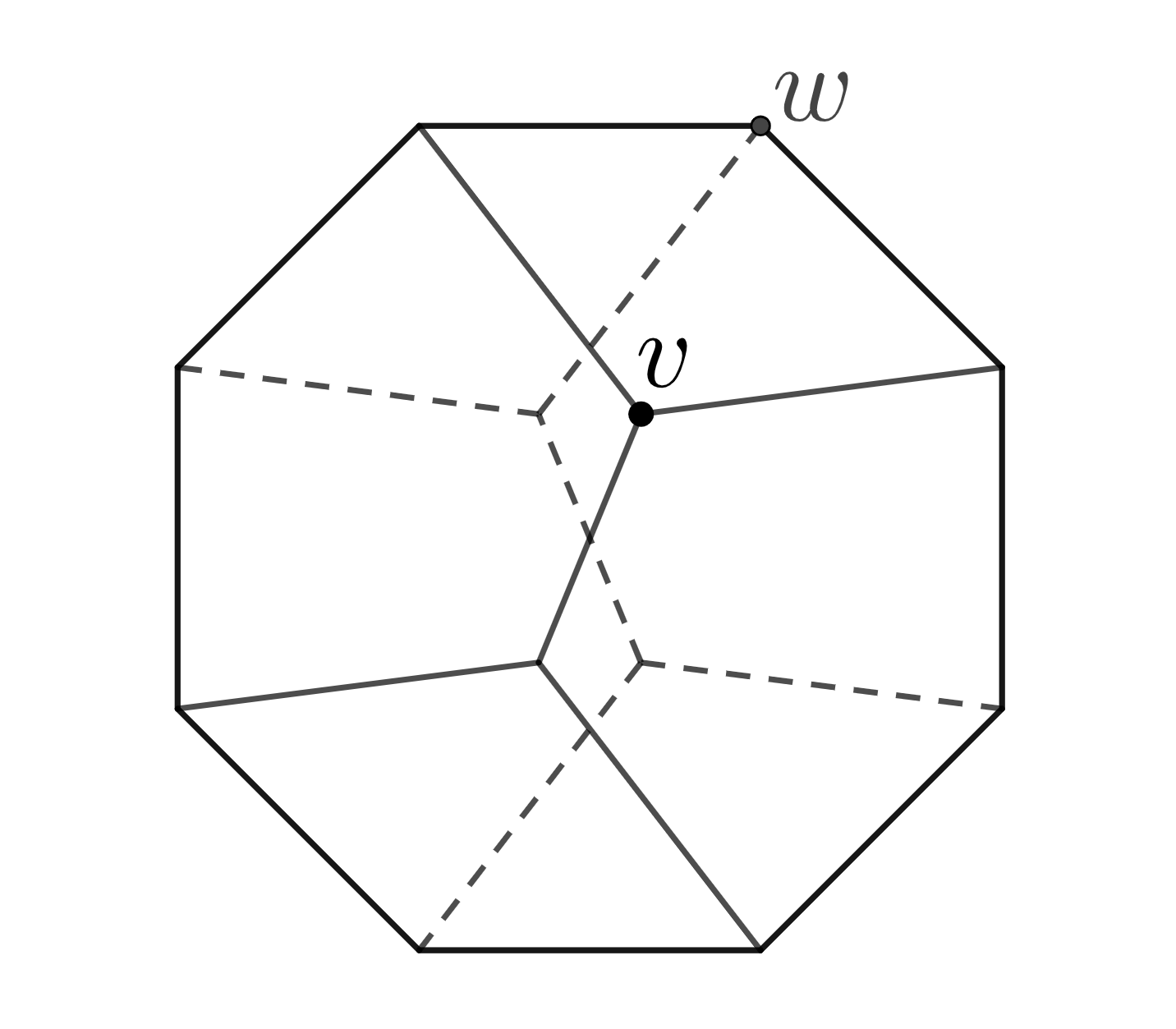}
\end{minipage}
\begin{minipage}[t]{0.45\linewidth}
\centering
\includegraphics[height=4cm]{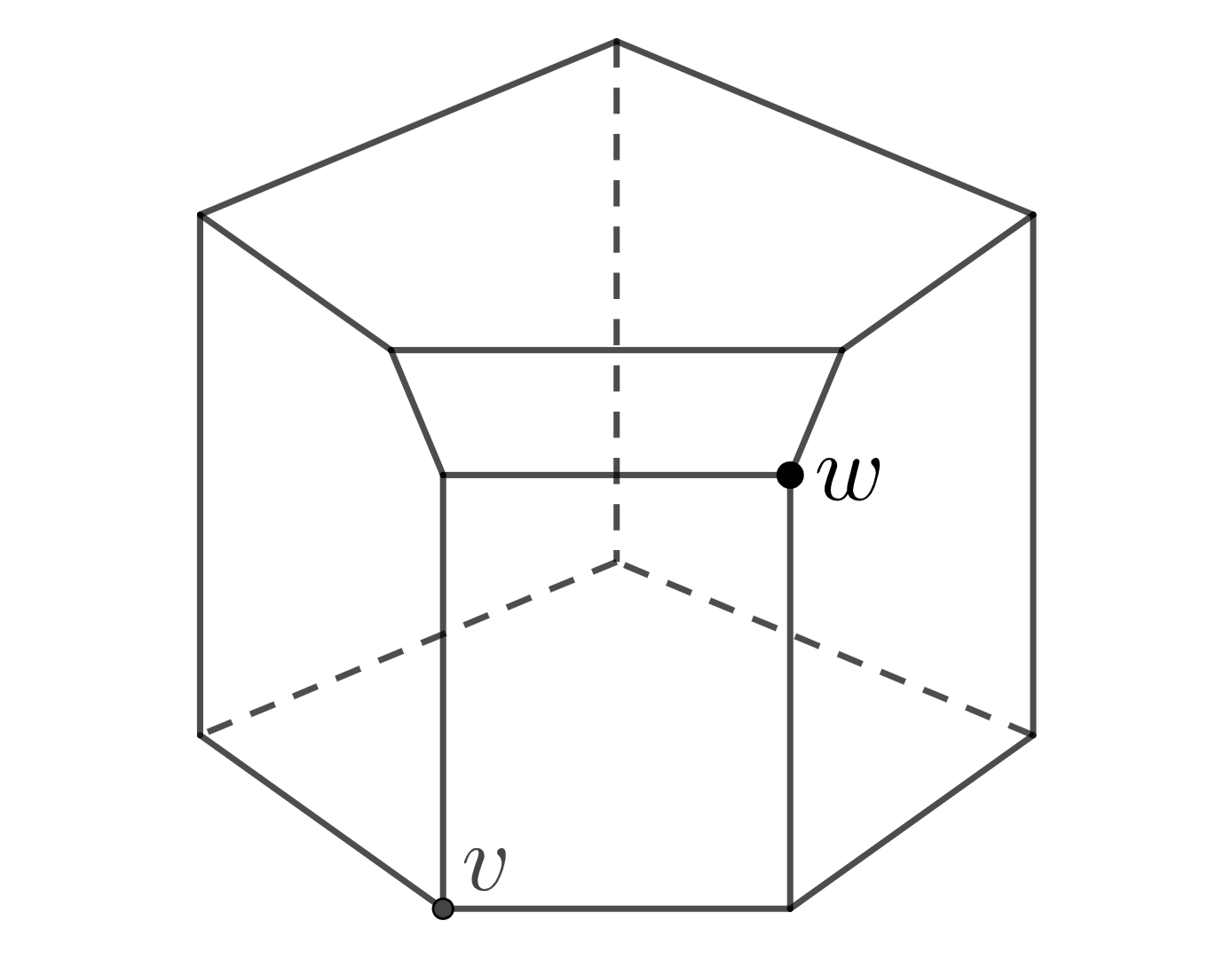}
\end{minipage}
\centering
\caption{$Q$ as an edge cut of $C_{2}(5) \times I$}
\label{Q3 as an edge cut}
\end{figure}

\noindent The classification above enables us to give a totally combinatorial characterization of string quasitoric manifolds over $n$-dimensional simple polytopes with the number of facets $f_{n-1}(P) \leq 2n+2$. 

\subsection{$f_{n-1}(P)=2n$} \label{2n}

Label the facets of $I^{n}$ such that $F_{i} \cap F_{n+i} = \emptyset$ for $1 \leq i \leq n$ (see Figure \ref{I3 and its dual with label} for $n=3$ case). Choose the basis of $H^{4}(M(I^{n},\Lambda))$ as $\{v_{i}v_{j}\}_{n+1 \leq i < j \leq 2n}$ and write $p_{1}(M)=p_{1}(M(I^{n},\Lambda))=\sum_{n+1 \leq i < j \leq 2n} c_{i,j} v_{i}v_{j}$. 

\begin{figure}[htb]
\begin{minipage}[t]{0.45\linewidth}
\centering
\includegraphics[height=4cm]{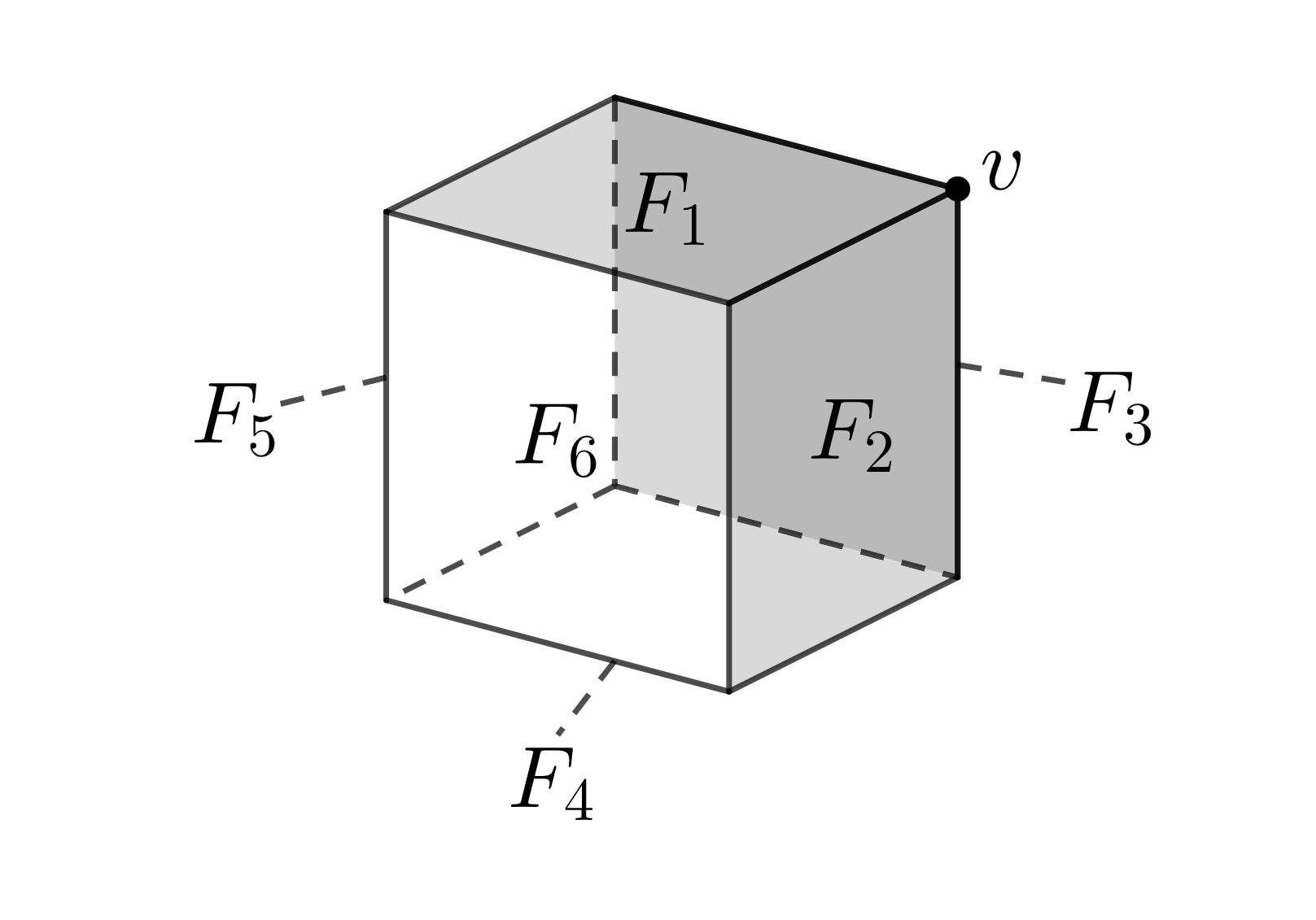}
\end{minipage}
\begin{minipage}[t]{0.45\linewidth}
\centering
\includegraphics[height=4cm]{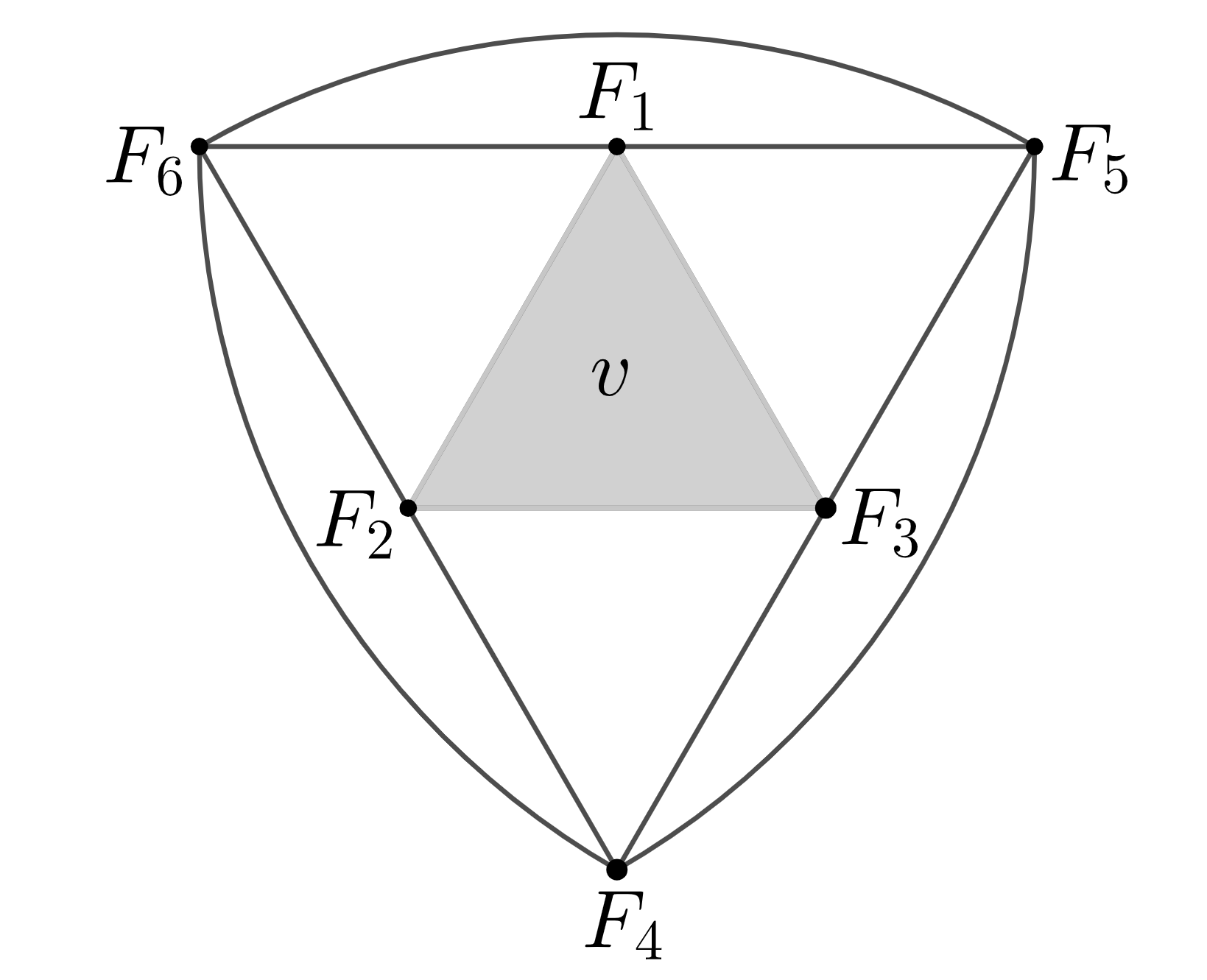}
\end{minipage}
\centering
\caption{$I^{3}$ and its dual with label}
\label{I3 and its dual with label}
\end{figure}

\begin{lemma} \label{coefficient 2n}
Write $\Lambda=(\la_{s,t})_{n \times 2n}$ and let $\rho_{i}=\sum_{k=1}^{n} \la_{k,i}^{2}+1$, $\rho_{j}=\sum_{k=1}^{n} \la_{k,j}^{2}+1$ and $\rho_{i,j}=\rho_{j,i}=2 \sum_{k=1}^{n} \la_{k,i}\la_{k,j}$ for $n+1 \leq i < j \leq 2n$. Then 
\[
c_{i,j}=-\la_{i-n,i}\la_{i-n,j} \rho_{i}-\la_{j-n,j}\la_{j-n,i} \rho_{j}+\rho_{i,j}. 
\]
\end{lemma}

\begin{proof}
By Proposition \ref{cohomology ring}, 
\[
\sum_{k=n+1}^{2n} \la_{i-n,k} v_{i}v_{k}=0 \qquad n+1 \leq i \leq 2n. 
\]
Since (\ref{nonsingular}) requires $\la_{i-n,i}=\pm 1$, direct computation based on Proposition \ref{characteristic class} yields 
\begin{align*}
p_{1}(M)=& \sum_{j=1}^{2n} v_{j}^{2} \\
=& \sum_{k=1}^{n} (\sum_{j=n+1}^{2n} \la_{k,j} v_{j})^{2}+\sum_{j=n+1}^{2n} v_{j}^{2} \\
=& \sum_{j=n+1}^{2n} (\sum_{k=1}^{n} \la_{k,j}^{2}+1) v_{j}^{2}+\sum_{n+1 \leq i < j \leq 2n} (2\sum_{k=1}^{n} \la_{k,i} \la_{k,j}) v_{i}v_{j} \\
=& -\sum_{n+1 \leq i < j \leq 2n} [(\sum_{k=1}^{n} \la_{k,i}^{2}+1) \la_{i-n,i}\la_{i-n,j} \\
& + (\sum_{k=1}^{n} \la_{k,j}^{2}+1) \la_{j-n,j}\la_{j-n,i}-2\sum_{k=1}^{n} \la_{k,i}\la_{k,j}] v_{i}v_{j} 
\end{align*}
\end{proof}

\begin{proposition} \label{2n coefficient}
A quasitoric manifold $M(I^{n},\Lambda)$ is string if and only if 
\[
\left\{
\begin{aligned}
& \sum_{k=1}^{n} \la_{k,i} \equiv 1 \pmod{2} & n+1 \leq i \leq 2n; \\
& -\la_{i-n,i}\la_{i-n,j} \rho_{i}-\la_{j-n,j}\la_{j-n,i} \rho_{j}+\rho_{i,j}=0 & \qquad n+1 \leq i < j \leq 2n. 
\end{aligned}
\right.
\]
\end{proposition}

The following algebraic lemma can be applied to further analyze string quasitoric manifolds over cube: 
\begin{lemma} \label{key lemma}
\emph{\cite[Theorem 6]{Do01}, see also \cite{CMS10}} Let $R$ be a commutative ring with unit 1 and $A$ be an $n \times n$ matrix with elements in $R$. Suppose every proper principal minor of $A$ is equal to 1. If $\mathrm{det}A=1$, then $A$ is conjugate to a unipotent upper triangular matrix. If $\mathrm{det}A=-1$, then $A$ is conjugate to the following matrix: 
\[
\begin{pmatrix}
1 & b_{1} & 0 & \cdots & 0 \\
0 & 1 & b_{2} & \cdots & 0 \\
\vdots & \vdots & \ddots & \ddots & \vdots \\
0 & 0 & \cdots & 1 & b_{n-1} \\
b_{n} & 0 & \cdots & 0 & 1
\end{pmatrix}
\]
with $\prod_{i=1}^{n} b_{i}=(-1)^{n} \cdot 2$. Here conjugacy is defined up to row and column permutations. 
\end{lemma}

For characteristic matrix $\Lambda=[\ \mathrm{I_{n}}\ |\ \Lambda_{*}\ ]$ of a quasitoric manifold over $I^{n}$, Dobrinskaya \cite{Do01} showed that $\Lambda_{*}$ is equivalent to a unipotent upper triangular matrix if and only if $M(I^{n},\Lambda)$ is weakly equivariantly homeomorphic to a Bott manifold, i.e., the total space $B^{2n}$ of a $\C P^{1}$-bundle tower: 
\[
B^{2n} \xrightarrow{\C P^{1}} B^{2n-2} \xrightarrow{\C P^{1}} \cdots \xrightarrow{\C P^{1}} B^{2} \xrightarrow{\C P^{1}} {pt}.
\]
Bott manifold may not be string (even spin) in general, but we have the following theorem in the other direction: 

\begin{theorem} \label{2n manifold}
Every string quasitoric manifold over $I^{n}$ is weakly equivariantly homeomorphic to a Bott manifold. 
\end{theorem}

\begin{proof}
Given a string quasitoric manifold $M(I^{n},\Lambda)$, it suffices to show that up to equivalence, all principal minors of $\Lambda_{*}$ are equal to 1. Principal minors with rank 1 are just $\la_{i-n,i}$ for $n+1 \leq i \leq 2n$ and they can be assumed to be 1 by sign permutation of columns. Consequently, principal minors with rank 2 are $1-\la_{i-n,j}\la_{j-n,i}$ for $n+1 \leq i < j \leq 2n$. If $1-\la_{i-n,j}\la_{j-n,i} \neq 1$, then $\la_{i-n,j}\la_{j-n,i}=2$ by (\ref{nonsingular}), leading to contradiction against Proposition \ref{2n coefficient}: 
\begin{align*}
|\la_{i-n,j} \rho_{i}+\la_{j-n,i} \rho_{j}| \geq & \sum_{k=1}^{n} (\la_{k,i}^{2}+\la_{k,j}^{2})+3 \\
\geq & 2 \sum_{k=1}^{n} |\la_{k,i}\la_{k,j}|+3 \\
> & |\rho_{i,j}|. 
\end{align*}

Now suppose there exists a principal minor with rank $s$ and value $-1$ such that all principal minors with rank less than $s$ are equal to 1. Then $s \geq 3$ and $\Lambda_{*}$ is equivalent to 
\[
\left(
\begin{array}{cccc|c}
1 & b_{1} & \cdots & 0 & \multirow{4}*{$*_{(n-s) \times s}$} \\
\vdots & \ddots & \ddots & \vdots & \\
0 & \cdots & 1 & b_{s-1} & \\
b_{s} & \cdots & 0 & 1 & \\
\hline
\la_{s+1,n+1} & \cdots & \la_{s+1,n+s-1} & \la_{s+1,n+s} & \multirow{3}*{$*_{(n-s) \times (n-s)}$} \\
\vdots & \vdots & \vdots & \vdots & \\
\la_{n,n+1} & \cdots & \la_{n,n+s-1} & \la_{n,n+s} & 
\end{array}
\right)
\]
with $\prod_{k=1}^{s} b_{k}=(-1)^{s} \cdot 2$. Moreover, we can assume that $(b_{1}, b_{2}, \dots, b_{s})=(-1, -1, \dots, -2)$ by taking necessary conjugations and sign permutations. Then by Proposition \ref{2n coefficient}: 
\[
2\sum_{k=1}^{n} \la_{k,i}\la_{k,j}= 
\left\{
\begin{aligned}
& (-1)(\sum_{k=1}^{n} \la_{k,i}^{2}+1) & \qquad n+1 \leq i \leq n+s-1, j=i+1; \\
& (-2)(\sum_{k=1}^{n} \la_{k,n+s}^{2}+1) & i=n+1, j=n+s; \\
& 0 & \mathrm{otherwise}. 
\end{aligned}
\right.
\]
Taking the sum for all $n+1 \leq i < j \leq n+s$, we have 
\begin{equation*}
2\sum_{k=s+1}^{n}\sum_{n+1 \leq i < j \leq n+s} \la_{k,i}\la_{k,j}-2(s+1) = -\sum_{k=s+1}^{n}(\sum_{i=n+1}^{n+s-1} \la_{k,i}^{2}+2\la_{k,n+s}^{2})-3(s+2), 
\end{equation*}
leading to the following contradiction: 
\begin{equation*}
\sum_{k=s+1}^{n} [(\sum_{i=n+1}^{n+s} \la_{k,i})^{2}+\la_{k,n+s}^{2}]=-(s+4)<0. 
\end{equation*}
\end{proof}

\begin{remark} \label{Cartesian product}
A characteristic pair $(P,\Lambda)$ with $P=P_{1}^{n_{1}} \times P_{2}^{n_{2}}$ induces quasitoric manifolds $M_{1}=M(P_{1}^{n_{1}},\Lambda_{1})$ and $M_{2}=M(P_{2}^{n_{2}},\Lambda_{2})$ such that $\Lambda_{1}=[\ \mathrm{I_{n_{1}}}\ |\ \Lambda_{1*}\ ]$, $\Lambda_{2}=[\ \mathrm{I_{n_{2}}}\ |\ \Lambda_{2*}\ ]$ and $\Lambda$ is equivalent to 
\[
\left(
\begin{array}{c|c|c|c}
\mathrm{I_{n_{1}}} & 0 & \Lambda_{1*} & *_{12} \\
\hline
0 & \mathrm{I_{n_{2}}} & *_{21} & \Lambda_{2*} 
\end{array}
\right).
\]
$M(P,\Lambda)$ is an $M_{2}$ bundle over $M_{1}$ (resp. $M_{1}$ bundle over $M_{2}$) if and only if $*_{12}$ (resp. $*_{21}$) is a zero submatrix. And $M(P,\Lambda)=M_{1} \times M_{2}$ if and only if both $*_{12}$ and $*_{21}$ are zero submatrices. 

Suppose $f_{n_{1}-1}(P_{1})=m_{1}$ and $f_{n_{2}-1}(P_{2})=m_{2}$, then one can choose the basis of $H^{4}(M(P,\Lambda))$ to be $\mathcal{B}_{1} \sqcup \mathcal{B}_{2} \sqcup \mathcal{B}_{12}$, where $\mathcal{B}_{1}$ (resp. $\mathcal{B}_{2}$) corresponds to basis of $H^{4}(M_{1})$ (resp. $H^{4}(M_{2})$) and $\mathcal{B}_{12}$ includes all non-vanishing $v_{i} v_{j}$ for $n_{1}+n_{2}+1 \leq i \leq m_{1}+n_{2}$, $m_{1}+n_{2}+1 \leq j \leq m_{1}+m_{2}$. 

When $P_{2}^{n_{2}}=I^{n_{2}}$, the proof above still works for $M_{2}$, i.e., if $M(P,\Lambda)$ is string, then $\Lambda_{2*}$ is equivalent to a unipotent upper triangular matrix and $M_{2}$ is weakly equivariantly homeomorphic to a Bott manifold. Note that in this case, $M_{2}$ may not be string (even spin). 
\end{remark}

\begin{remark}
Every Bott manifold has vanishing Stiefel-Whitney numbers and Pontryagin numbers \cite{Lu17}. Thus, every string quasitoric manifold over cube bounds non-equivariantly in $\Omega_{*}^{O}$ and $\Omega_{*}^{SO}$.
\end{remark}

\begin{example}
For string quasitoric manifolds $M=M(I^3,\Lambda)$, $\Lambda_{*}$ is equivalent to 
\[
\begin{pmatrix}
1 & 0 & x \\
0 & 1 & y \\
0 & 0 & 1
\end{pmatrix}\ 
\mathrm{or}\ 
\begin{pmatrix}
1 & 2a & ab \\
0 & 1 & b \\
0 & 0 & 1
\end{pmatrix}
\]
with $x \equiv y \pmod{2}$ and $ab \equiv b \pmod{2}$. Since $\{ \Lambda_{*}=\left( \begin{smallmatrix} 1 & 0 & 2t \\ 0 & 1 & 0 \\ 0 & 0 & 1\end{smallmatrix} \right) |\ t \in \mathbb{N} \}$ are not equivalent to each other, there are countably many weakly equivariant homeomorpism classes. On the other hand, note that 
\[
H^{*}(M) \cong 
\left\{
\begin{aligned}
& \Z[\alpha, \beta, \gamma]/\langle \alpha^{2}, \beta^{2}, \gamma^{2}\rangle & \quad x \equiv 0 \pmod{2}\ \mathrm{or}\ ab \equiv 0 \pmod{2}; \\
& \Z[\alpha, \beta, \gamma]/\langle \alpha(\alpha+\gamma), \beta(\beta+\gamma), \gamma^{2}\rangle & x \equiv 1 \pmod{2}\ \mathrm{or}\ ab \equiv 1 \pmod{2},
\end{aligned}
\right.
\]
and $6$-dimensional string quasitoric manifolds admit strong cohomology rigidity \cite{Ju73, Wall66}, there are only 2 homeomorphism classes in this case. But in general, one can not get similar counting results at homeomorphism level since cohomology rigidity problem is still open in higher dimensions (see \cite{CMS11} for more details). 
\end{example}

\vspace{0.5cm}

In addition, we shall see that string property for quasitoric manifolds over $I^{n} \# P^{n}$ is related with decomposition via equivariant connected sum operation.

\begin{definition} \label{equivariant connected sum definition} (see \cite{BPR07} and references given there) 
Given $2n$-dimensional quasitoric manifolds $M(P_{L},\Lambda_{L})$ and $M(P_{R},\Lambda_{R})$, write $\Lambda_{L}=[\ \mathrm{I_{n}}\ |\ \Lambda_{L*}\ ]$, $\Lambda_{R}=[\ \mathrm{I_{n}}\ |\ \Lambda_{R*}\ ]$ and let $w_{L}, w_{R}$ be initial vertices. The \emph{equivariant connected sum} at $w_{L}, w_{R}$ is denoted by $M(P_{L},\Lambda_{L}) \widetilde{\#}_{w_{L}, w_{R}} M(P_{R},\Lambda_{R})$ and defined to be a quasitoric manifold $M(P,\Lambda)$, where $P=P_{L} \# P_{R}$ is the connected sum of $P_{L}$ and $P_{R}$ at $w_{L}, w_{R}$ (see Figure \ref{cube connected sum} as an example) and 
\[
\Lambda=
\left(
\begin{array}{c|c|c}
\Lambda_{L*} & \mathrm{I_{n}} & \Lambda_{R*}
\end{array}
\right).
\]

\begin{figure}[htb]
\centering
\includegraphics[height=4cm]{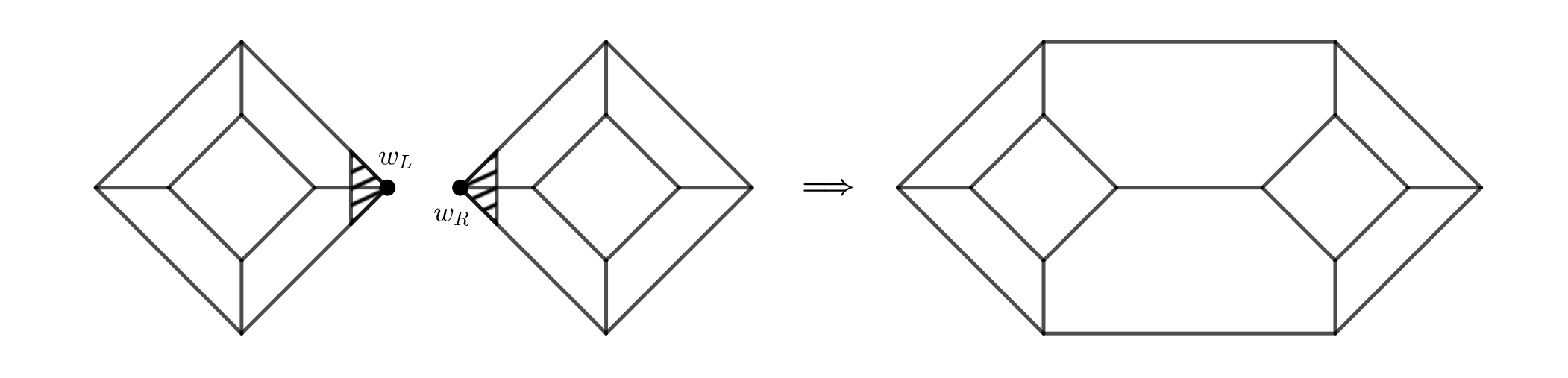}
\caption{Connected sum of two cubes}
\label{cube connected sum}
\end{figure}

Note that $\Lambda$ is NOT in refined form here and one can choose basis of $H^{4}(M(P,\Lambda))$ to be $\mathcal{B}_{1} \sqcup \mathcal{B}_{2}$, where $\mathcal{B}_{1}$ (resp. $\mathcal{B}_{2}$) corresponds to basis of $H^{4}(M(P_{L},\Lambda_{L}))$ (resp. $H^{4}(M(P_{R},\Lambda_{R}))$). In this way, it is evident that 
\[
\left\{
\begin{aligned}
& w_{2}(M(P,\Lambda))=0 \Leftrightarrow w_{2}(M(P_{L},\Lambda_{L}))=w_{2}(M(P_{R},\Lambda_{R}))=0; \\
& p_{1}(M(P,\Lambda))=0 \Leftrightarrow p_{1}(M(P_{L},\Lambda_{L}))=p_{1}(M(P_{R},\Lambda_{R}))=0. 
\end{aligned}
\right. 
\]
Also note that different choices of initial vertices may lead to different results, but one can simplify the notation as $M(P_{L},\Lambda_{L}) \widetilde{\#} M(P_{R},\Lambda_{R})$ when no confusion can arise. 
\end{definition} 

\begin{theorem} \label{equivariant connected sum theorem}
$M(I^{n} \# P^{n},\Lambda)$ is string if and only if it is weakly equivariantly homeomorphic to $M(I^{n},\Lambda_{L}) \widetilde{\#} M(P^{n},\Lambda_{R})$ with both $M(I^{n},\Lambda_{L})$ and $M(P^{n},\Lambda_{R})$ string. 
\end{theorem}

\begin{proof}
By explanation above, it remains to prove the necessity. Let $\mathcal{F}(I^{n})=\{F'_{i}\}_{i=1}^{2n}$, $\mathcal{F}(P^{n})=\{F''_{i}\}_{i=1}^{m}$ denote the facet sets and suppose connected sum is taken at $w_{L}=\cap_{i=n+1}^{2n} F'_{i},\ w_{R}=\cap_{i=1}^{n} F''_{i}$. Label the facets of $I^{n} \# P^{n}$ such that $F_{i}=F'_{i}$ for $1 \leq i \leq n$, $F_{i}=F''_{i-n}$ for $2n+1 \leq i \leq n+m$ and $F_{j}$ is formed by $F'_{j}$ together with $F''_{j-n}$ for $n+1 \leq j \leq 2n$. Choose the initial vertex as $w=\cap_{i=1}^{n} F_{i}$, then 
\[
\Lambda=
\left(
\begin{array}{c|c|c}
\mathrm{I_{n}} & A & \Lambda_{P^{n}*}
\end{array}
\right), 
\]
and it suffices to show $\mathrm{det}A=1$ up to equivalence. 

Take $\{v_{i}\}_{i=n+1}^{n+m}$ as the basis of $H^{2}(M(I^{n} \# P^{n},\Lambda))$, then $H^{4}(M(I^{n} \# P^{n},\Lambda))$ includes 5 types of elements: 
\begin{align*}
& T_{1}=\{v_{i}^{2}\}_{n+1 \leq i \leq 2n}; \\
& T_{2}=\{v_{i} v_{i'}\}_{n+1 \leq i < i' \leq 2n}; \\
& T_{3}=\{v_{i} v_{j}\}_{n+1 \leq i \leq 2n, 2n+1 \leq j \leq n+m}; \\
& T_{4}=\{v_{j}^{2}\}_{2n+1 \leq j \leq n+m}; \\
& T_{5}=\{v_{j} v_{j'}\}_{2n+1 \leq j < j' \leq n+m}. 
\end{align*}

Aside from vanishing elements in $T_{2}, T_{3}$ and $T_{5}$, ideal $\mathcal{I}$ contains 2 types of relations in $H^{4}(M(I^{n} \# P^{n},\Lambda))$: 
\[
\left\{
\begin{aligned}
& v_{i-n}v_{i}=0 & n+1 \leq i \leq 2n; \\
& v_{i-n}v_{j}=0 & \qquad n+1 \leq i \leq 2n, 2n+1 \leq j \leq n+m. 
\end{aligned}
\right. 
\]
Clearly, the former can kill all elements in $T_{1}$ while the latter can kill some elements in $T_{3}, T_{4}$ and $T_{5}$. In particular, elements in $T_{2}$ are only involved in the former type of relations. Thus, they can be chosen as part of basis of $H^{4}(M(I^{n} \# P^{n},\Lambda))$ and corresponding coefficients are the same as cube case. Similar to the argument in Theorem \ref{2n manifold}, one can show that up to equivalence, every proper principal minor of $A$ is equal to 1 and 
\[
A=(a_{i,j})_{n \times n}=
\left(
\begin{array}{c|c}
1 & *_{1 \times (n-1)} \\
\hline
*_{(n-1) \times 1} & B
\end{array}
\right), 
\]
where $B$ is a unipotent upper triangular matrix. If $a_{n,1}=0$, then expansion by minors on the last row induces $\mathrm{det}A=1$. Otherwise, since the coefficient of $v_{n+1}v_{n+k}$ is equal to $a_{n,1}^{2}a_{1,k}$ for $2 \leq k \leq n$, one gets $a_{1,k}=0$ for $2 \leq k \leq n$, leading to $\mathrm{det}A=\mathrm{det}B=1$. 
\end{proof}

\begin{remark}
String condition is essential to the validity of Theorem \ref{equivariant connected sum theorem}. Let $M$ be the quasitoric manifold over $I^{3} \# I^{3}$ (see Figure \ref{cube connected sum 1}) with characteristic matrix 
\[
\Lambda=
\begin{pmatrix}
1 & 0 & 0 & 2 & 2 & 3 & 1 & 2 & 2 \\
0 & 1 & 0 & 0 & 1 & 1 & 0 & 0 & 1 \\
0 & 0 & 1 & 1 & 0 & 1 & 0 & 1 & 0 
\end{pmatrix}. 
\] 
It is easy to verify that $M$ is spin but not string. And it is indecomposable via equivariant connected sum since $\mathrm{det} \left( \begin{smallmatrix} 0 & 0 & 3 \\ 1 & 0 & 1 \\ 0 & 1 & 1 \end{smallmatrix} \right)=3 \neq \pm 1$. 

\begin{figure}[htb]
\centering
\includegraphics[height=4cm]{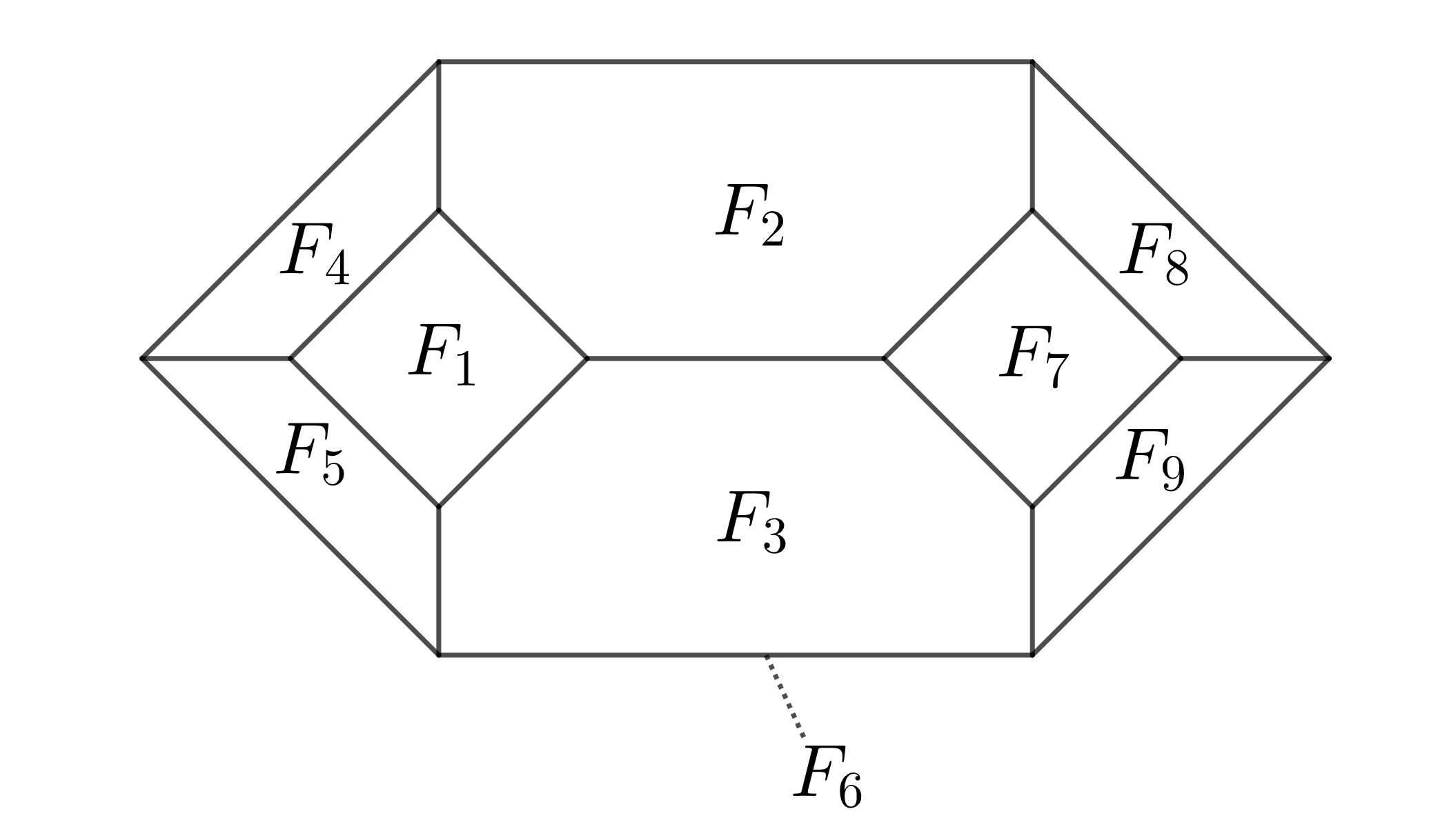}
\caption{$I^{3} \# I^{3}$ with label}
\label{cube connected sum 1}
\end{figure}

\end{remark}

\subsection{$f_{n-1}(P)=2n+1$} \label{2n+1}

As shown in Figure \ref{2n+1facets}, label the facets of $P=C_{2}(5) \times I^{n-2}$ such that 
\[
F_{i}=
\left\{
\begin{aligned}
& f_{i} \times I^{n-2} & 1 \leq i \leq 5; \\
& C_{2}(5) \times I^{i-6} \times a_{i-5} \times I^{n+3-i} & 6 \leq i \leq n+3; \\
& C_{2}(5) \times I^{i-n-4} \times b_{i-n-3} \times I^{2n+1-i} & \qquad n+4 \leq i \leq 2n+1. 
\end{aligned}
\right.
\]

\begin{figure}[htb]
\centering
\includegraphics[height=3.5cm]{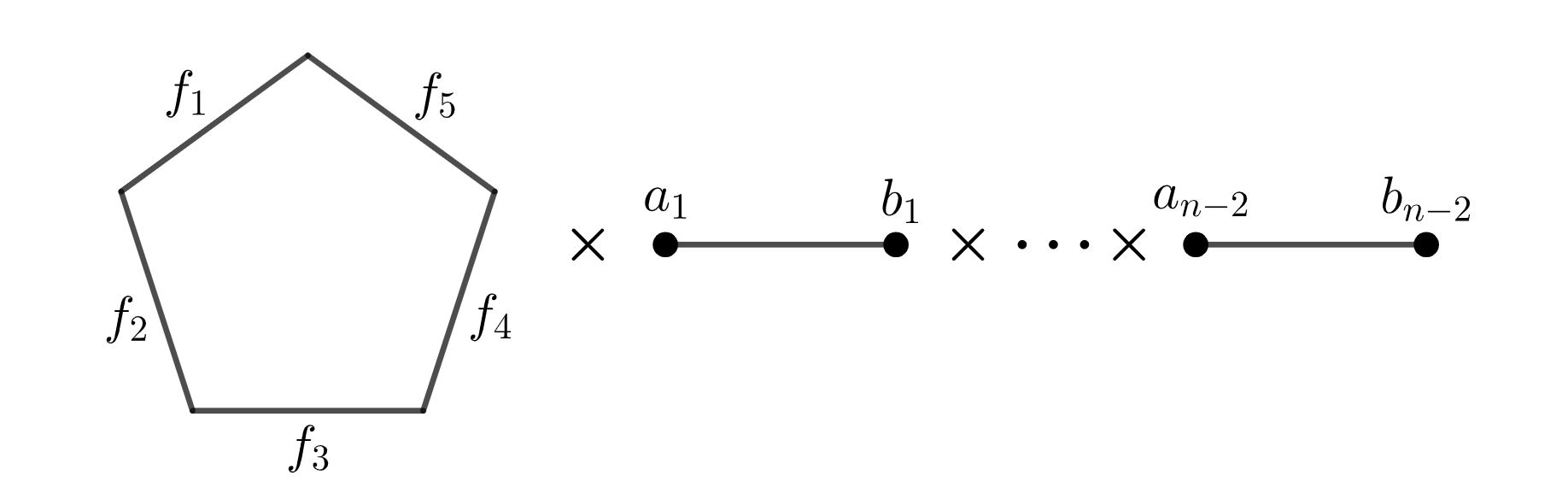}
\caption{$P=C_{2}(5) \times I^{n-2}$}
\label{2n+1facets}
\end{figure}

\noindent Let the initial vertex be intersection of $F_{1}, F_{2}$ and $\{F_{i}\}_{i=6}^{n+3}$, then $\Lambda$ is equivalent to 
\[
\left(
\begin{array}{c|ccc|c|cccc}
\multirow{2}*{$\mathrm{I}_{2}$} & 1 & \la_{1,4} & \la_{1,5} & \multirow{2}*{0} & \multicolumn{4}{c}{\multirow{2}*{$*_{2 \times (n-2)}$}} \\
& \la_{2,3} & \la_{2,4} & 1 & & \\
\hline
\multirow{4}*{0} & \multicolumn{3}{c|}{\multirow{4}*{$*_{(n-2) \times 3}$}} & \multirow{4}*{$\mathrm{I_{n-2}}$} & 1 & \la_{3,n+5} & \cdots & \la_{3,2n+1} \\
& & & & & \la_{4,n+4} & 1 &\cdots & \la_{4,2n+1} \\
& & & & & \vdots & \vdots & \ddots & \vdots \\
& & & & & \la_{n,n+4} & \la_{n,n+5} & \cdots & 1
\end{array}
\right).
\]
Fix the corresponding basis of $H^{2}(M(P,\Lambda))$ as $\{v_{3}, v_{4}, v_{5}\}$ together with $\{v_{i}\}_{i=n+4}^{2n+1}$, then $v_{4}v_{5}$, $\{v_{i}v_{j}\}_{3 \leq i \leq 5, n+4 \leq j \leq 2n+1}$ and $\{v_{i}v_{j}\}_{n+4 \leq i < j \leq 2n+1}$ form a basis of $H^{4}(M(P,\Lambda))$. 

Combine arguments in the proof of Lemma \ref{coefficient prism} and \ref{coefficient 2n}, we are led to the following formulas for corresponding coefficients: 

\begin{lemma} \label{coefficient 2n+1}
Using notations introduced in Lemma \ref{coefficient 2n}, 
\begin{align*}
& c_{i,j}=-\la_{i-n-1,j}\rho_{i}-\la_{j-n-1,i}\rho_{j}+\rho_{i,j} & \qquad n+4 \leq i < j \leq 2n+1; \\
& c_{3,i}=-\la_{1,i}\rho_{3}-\la_{i-n-1,3}\rho_{i}+\rho_{3,i} & n+4 \leq i \leq 2n+1; \\
& c_{5,i}=-\la_{2,i}\rho_{5}-\la_{i-n-1,5}\rho_{i}+\rho_{5,i} & n+4 \leq i \leq 2n+1. 
\end{align*}
Moreover, similar to Lemma \ref{dimP=2 coefficient}, set $\Delta_{i,j}=\mathrm{det} \left( \begin{smallmatrix} \la_{1,i} & \la_{1,j} \\ \la_{2,i} & \la_{2,j} \end{smallmatrix}\right)$ for $1 \leq i,j \leq 2n+1$ and $l_{i}=\Delta_{i-1,i} \cdot \Delta_{i,i+1} \cdot \Delta_{i+1,i-1}$ with subscripts taken modulo 5. Then 
\[
c_{4,5}=\Delta_{4,5} \sum_{k=1}^{5} (l_{k} \rho_{k}+\Delta_{k,k+1} \rho_{k,k+1}), 
\]
and for $ n+4 \leq i \leq 2n+1$, 
\[
c_{4,i}=-\Delta_{3,4}\Delta_{3,i} \rho_{4}-\la_{i-n-1,4} \rho_{i}+\rho_{4,i}+\Delta_{4,i} (l_{3} \rho_{3}+\Delta_{3,4} \rho_{3,4}). 
\]
\end{lemma}

\begin{proposition} \label{2n+1 coefficient}
A quasitoric manifold $M(C_{2}(5) \times I^{n-2},\Lambda)$ is string if and only if
\[
\sum_{k=1}^{n} \la_{k,i} \equiv 1 \pmod{2} \qquad 3 \leq i \leq 5; n+4 \leq i \leq 2n+1, 
\]
and all 5 types of coefficients listed in Lemma \ref{coefficient 2n+1} vanish. 
\end{proposition}

\begin{remark}
In general, string quasitoric manifolds over $P=C_{2}(s) \times I^{n-2}$ $(s \geq 4)$ can be characterized in a similar way. In particular, such manifolds always exist when $s \equiv 0 \pmod{2}$ while they only exist for $n \geq 4$ when $s \equiv 1 \pmod{2}$. 
\end{remark}

\subsection{$f_{n-1}(P)=2n+2$} \label{2n+2}

Arguments in Subsection \ref{2n+1} apply to the case of $C_{2}(6) \times I^{n-2}$ and $C_{2}(5) \times C_{2}(5) \times I^{n-4}$. The details are left to the reader. 

\vspace{0.5cm}

When $P=Q \times I^{n-3}$, label the facets such that $F_{i}=f_{i} \times I^{n-3}$ for $1 \leq i \leq 8$ (see Figure \ref{Q3 and its dual with label}) and $F_{i} \cap F_{n-3+i}=\emptyset$ for $9 \leq i \leq n+5$. 

\begin{figure}[htb]
\begin{minipage}[t]{0.45\linewidth}
\centering
\includegraphics[height=4cm]{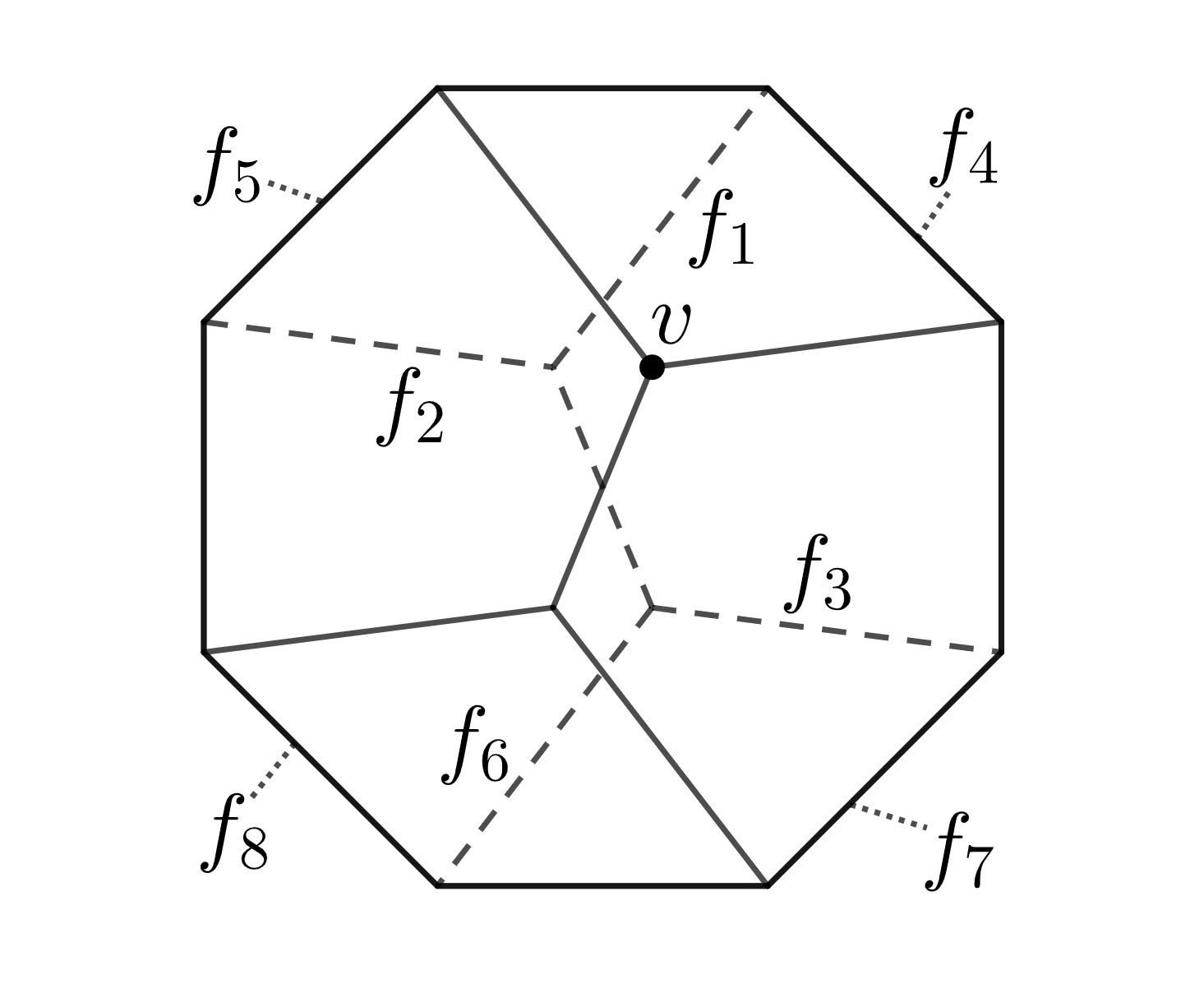}
\end{minipage}
\begin{minipage}[t]{0.45\linewidth}
\centering
\includegraphics[height=4cm]{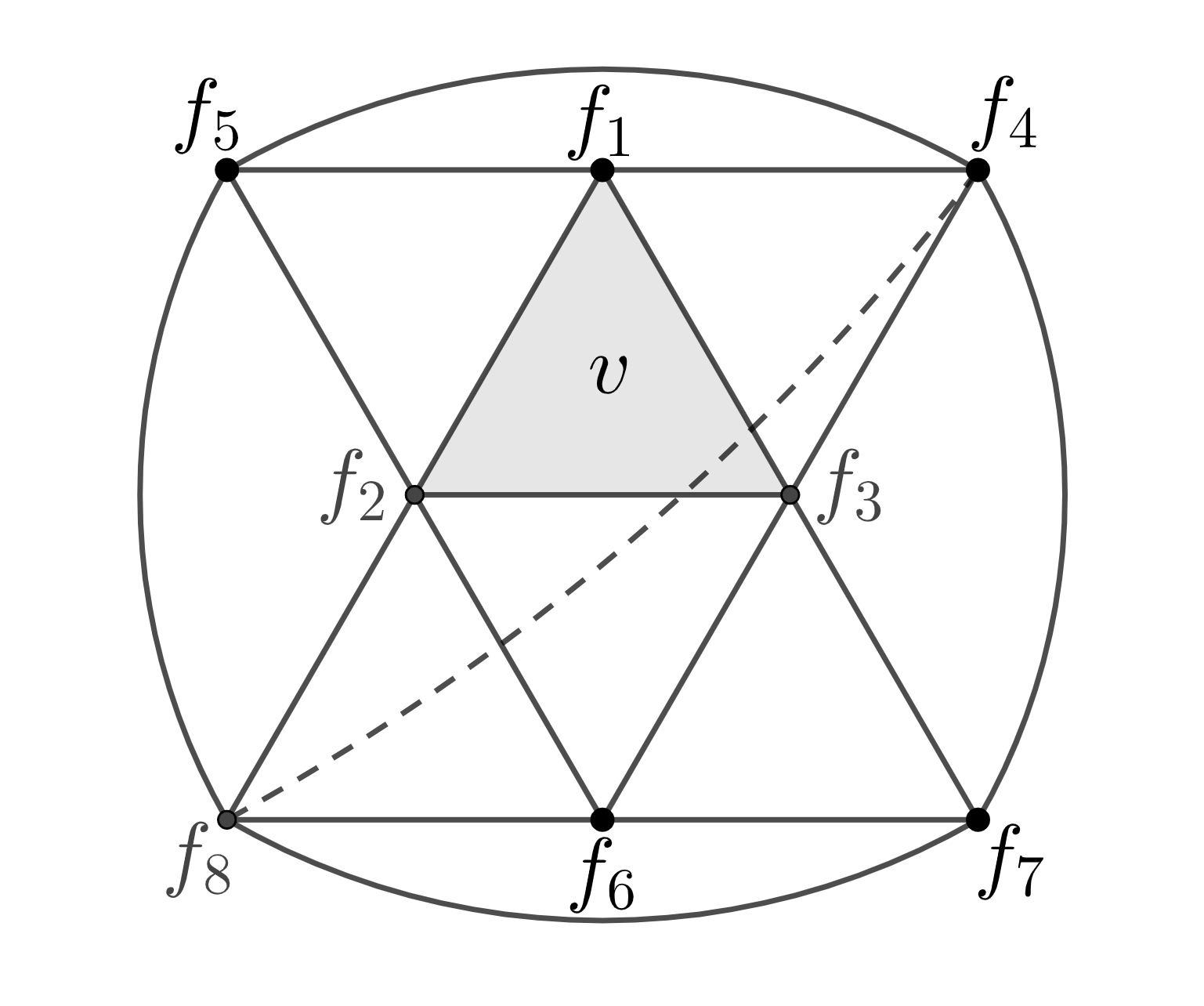}
\end{minipage}
\caption{$Q^{3}$ and its dual with label}
\label{Q3 and its dual with label}
\end{figure}

\noindent Let the initial vertex be the intersection of $F_{1}, F_{2}, F_{3}$ and $\{F_{i}\}_{i=9}^{n+5}$, then $\Lambda$ is equivalent to 
\[
\left(
\begin{array}{c|ccccc|c|cccc}
\multirow{3}*{$\mathrm{I}_{3}$} & \la_{1,4} & \la_{1,5} & 1 & \la_{1,7} & \la_{1,8} & \multirow{3}*{0} & \multicolumn{4}{c}{\multirow{3}*{$*_{3 \times (n-3)}$}} \\
& 1 & \la_{2,5} & \la_{2,6} & \la_{2,7} & \la_{2,8} & & \\
& \la_{3,4} & 1 & \la_{3,6} & \la_{3,7} & \la_{3,8} & & \\
\hline
\multirow{4}*{0} & \multicolumn{5}{c|}{\multirow{4}*{$*_{(n-3) \times 5}$}} & \multirow{4}*{$\mathrm{I_{n-3}}$} & 1 & \la_{4,n+7} & \cdots & \la_{4,2n+2} \\
& & & & & & & \la_{5,n+6} & 1 &\cdots & \la_{5,2n+2} \\
& & & & & & & \vdots & \vdots & \ddots & \vdots \\
& & & & & & & \la_{n,n+6} & \la_{n,n+7} & \cdots & 1
\end{array}
\right).
\]
Fix the corresponding basis of $H^{2}(M(P,\Lambda))$ as $\{v_{i}\}_{i=4}^{8}$ together with $\{v_{i}\}_{i=n+6}^{2n+2}$, then the following 5 classes of elements form a basis of $H^{4}(M(P,\Lambda))$: 
\begin{align*}
& C_{1}=\{v_{i}v_{j}\}_{n+6 \leq i < j \leq 2n+2}; \\
& C_{2}=\{v_{4}v_{i}, v_{5}v_{i}, v_{6}v_{i}\}_{n+6 \leq i \leq 2n+2}; \\
& C_{3}=\{v_{4}v_{5}, v_{5}v_{8}, v_{7}v_{4}\}; \\
& C_{4}=\{v_{7}v_{i}, v_{8}v_{i}\}_{n+6 \leq i \leq 2n+2}; \\
& C_{5}=\{v_{4}v_{8}, v_{7}v_{8}\}. 
\end{align*}

In order to get relatively simple and explicit formulas for corresponding coefficients in the expression of $p_{1}(M(P,\Lambda))$, more generalized notations are needed: 

\begin{notation} \label{modified notation}
Set $\Delta_{i,j,k}=\mathrm{det} \left( \begin{smallmatrix} \la_{1,i} & \la_{1,j} & \la_{1,k} \\ \la_{2,i} & \la_{2,j} & \la_{2,k} \\ \la_{3,i} & \la_{3,j} & \la_{3,k} \end{smallmatrix} \right)$ for $1 \leq i, j, k \leq 2n+2$. Let $g_{1}$ be a function over integers with period 4 and $g_{2}, g_{3}$ be functions over integers with period 5, such that 
\[
g_{1}(a)=\left\{
\begin{aligned}
& 2 & \quad a=1; \\
& 3 & a=2; \\
& 4 & a=3; \\
& 5 & a=4; 
\end{aligned}
\right.
\qquad g_{2}(a)=\left\{
\begin{aligned}
& 3 & \quad a=1; \\
& 1 & a=2; \\
& 5 & a=3; \\
& 8 & a=4; \\
& 6 & a=5; 
\end{aligned}
\right.
\qquad g_{3}(a)=\left\{
\begin{aligned}
& 1 & \quad a=1; \\
& 2 & a=2; \\
& 6 & a=3; \\
& 7 & a=4; \\
& 4 & a=5. 
\end{aligned}
\right.
\]
Note that values of $g_{i}$ are nothing but labels of facets adjacent to $f_{i}$ in $Q$ with counter-clockwise order (see Figure \ref{Q3 and its dual with label}). Then set $\widetilde{\rho}_{i,j}=\rho_{g_{i}(j)}$, $\widetilde{\rho}_{i,j,k}=\rho_{g_{i}(j),g_{i}(k)}$, $\widetilde{\Delta}_{i,j,k}=\Delta_{i,g_{i}(j),g_{i}(k)}$ and $\widetilde{l}_{i,j}=\widetilde{\Delta}_{i,j-1,j} \cdot \widetilde{\Delta}_{i,j,j+1} \cdot \widetilde{\Delta}_{i,j+1,j-1}$ for $1 \leq i \leq 3$ and $1 \leq j, k \leq 2n+2$. 
\end{notation}

\begin{lemma} \label{coefficient 2n+2B}
For classes $C_{1}$ and $C_{2}$: 
\[
\begin{aligned}
& c_{i,j}=-\la_{i-n-2,j} \rho_{i}-\la_{j-n-2,i} \rho_{j}+\rho_{i,j} & \qquad n+6 \leq i < j \leq 2n+2; \\
& c_{4,i}=-\la_{2,i} \rho_{4}-\la_{i-n-2,4} \rho_{i}+\rho_{4,i} & n+6 \leq i \leq 2n+2; \\
& c_{5,i}=-\la_{3,i} \rho_{5}-\la_{i-n-2,5} \rho_{i}+\rho_{5,i} & n+6 \leq i \leq 2n+2; \\
& c_{6,i}=-\la_{1,i} \rho_{6}-\la_{i-n-2,6} \rho_{i}+\rho_{6,i} & n+6 \leq i \leq 2n+2. 
\end{aligned}
\]
For class $C_{3}$: 
\[
\begin{aligned}
& c_{4,5}=\Delta_{1,4,5} \sum_{k=1}^{4} (\widetilde{l}_{1,k} \widetilde{\rho}_{1,k}+\widetilde{\Delta}_{1,k,k+1} \widetilde{\rho}_{1,k,k+1}); \\
& c_{5,8}=\Delta_{2,5,8} \sum_{k=1}^{5} (\widetilde{l}_{2,k} \widetilde{\rho}_{2,k}+\widetilde{\Delta}_{2,k,k+1} \widetilde{\rho}_{2,k,k+1}); \\
& c_{7,4}=\Delta_{3,7,4} \sum_{k=1}^{5} (\widetilde{l}_{3,k} \widetilde{\rho}_{3,k}+\widetilde{\Delta}_{3,k,k+1} \widetilde{\rho}_{3,k,k+1}). 
\end{aligned}
\]
For class $C_{4}$ with $n+6 \leq i \leq 2n+2$: 
\[
\begin{aligned}
& c_{7,i}=-\Delta_{3,6,7}\Delta_{3,6,i} \rho_{7}-\la_{i-n-2,7} \rho_{i}+\rho_{7,i}+\Delta_{3,7,i}(\widetilde{l}_{3,3} \widetilde{\rho}_{3,3}+\widetilde{\Delta}_{3,3,4} \widetilde{\rho}_{3,3,4}); \\
& c_{8,i}=-\Delta_{2,6,8}\Delta_{2,6,i} \rho_{8}-\la_{i-n-2,8} \rho_{i}+\rho_{8,i}+\Delta_{2,8,i}(\widetilde{l}_{2,5} \widetilde{\rho}_{2,5}+\widetilde{\Delta}_{2,4,5} \widetilde{\rho}_{2,4,5}). 
\end{aligned}
\]
For class $C_{5}$: 
\[
\begin{aligned}
c_{4,8}= & -\la_{2,8} \rho_{4}-\Delta_{2,6,4}\Delta_{2,6,8} \rho_{8}+\rho_{4,8}+\Delta_{2,4,8}(\widetilde{l}_{2,5} \widetilde{\rho}_{2,5}+\widetilde{\Delta}_{2,4,5} \widetilde{\rho}_{2,4,5}); \\
c_{7,8}= & -\Delta_{3,6,7}\Delta_{3,6,8} \rho_{7}-\Delta_{2,6,7}\Delta_{2,6,8} \rho_{8}+\rho_{7,8} \\
& +\Delta_{3,7,8}(\widetilde{l}_{3,3} \widetilde{\rho}_{3,3}+\widetilde{\Delta}_{3,3,4} \widetilde{\rho}_{3,3,4})+\Delta_{2,7,8}(\widetilde{l}_{2,5} \widetilde{\rho}_{2,5}+\widetilde{\Delta}_{2,4,5} \widetilde{\rho}_{2,4,5}). 
\end{aligned}
\]
\end{lemma}

\begin{proposition} \label{2n+2B coefficient}
A quasitoric manifold $M(Q \times I^{n-3},\Lambda)$ is string if and only if 
\[
\sum_{k=1}^{n} \la_{k,i} \equiv 1 \pmod{2} \qquad 4 \leq i \leq 8; n+6 \leq i \leq 2n+2, 
\]
and all coefficients listed in Lemma \ref{coefficient 2n+2B} vanish. 
\end{proposition}

\begin{example}
The quasitoric manifold $M(Q \times I^{2},\Lambda)$ is string if $\Lambda$ is equivalent to 
\[
\left(
\begin{array}{ccc|ccccc|cc|cc}
1 & 0 & 0 & 0 & 0 & 1 & 1 & 1 & 0 & 0 & 0 & 0 \\
0 & 1 & 0 & 1 & 0 & 0 & 1 & 0 & 0 & 0 & 0 & 0 \\
0 & 0 & 1 & 0 & 1 & 0 & 0 & 1 & 0 & 0 & 0 & 0 \\
\hline
0 & 0 & 0 & 2 & 0 & 2 & 2 & 0 & 1 & 0 & 1 & 0 \\
0 & 0 & 0 & 0 & 2 & 2 & 1 & 3 & 0 & 1 & 0 & 1 
\end{array}
\right)
\]
\end{example}

\section{Real analogue} \label{real analogue}
\setcounter{equation}{0}

A real version of quasitoric manifold called small cover, denoted by $M_{\R}(P,\Lambda)$, can be defined with $T^{n}$ replaced by $\Z_{2}^{n}$ in Canonical Construction in Section \ref{preliminaries} \cite{DJ91}. Suppose $P$ is an $n$-dimensional simple polytope with facet set $\mathcal{F}(P)=\{F_{i}\}_{i=1}^{m}$ and let $v_{i}$ be Poincar\'e dual of characteristic submanifold corresponding to $F_{i}$ for $1 \leq i \leq m$. Results parallel to Proposition \ref{Betti number} and \ref{cohomology ring} are valid with $\Z_{2}$-coefficients. As for characteristic classes, 
\[
w(M_{\R}(P,\Lambda))=\prod_{i=1}^{m} (1+v_{i}) \qquad p(M_{\R}(P,\Lambda))=1. 
\]
In particular, $w_{1}(M_{\R}(P,\Lambda))=\sum\limits_{i=1}^{m} v_{i}$, $w_{2}(M_{\R}(P,\Lambda))=\sum\limits_{1 \leq i < j \leq m} v_{i}v_{j}$ and $p_{1}(M_{\R}(P,\Lambda))=0$. 

Therefore, when $\Lambda$ is in refined form, $M_{\R}(P,\Lambda)$ is orientable if and only if all column sums of $\Lambda$ are odd. Moreover, string property and spin property are equivalent for small covers. Note that in the case of small cover, $w_{2}$ is the sum of square-free elements instead of square elements, and calculations are taken in $\Z_{2}$. Thus, it is not surprising that different results are obtained via simpler arguments. In the sequel, several statements are listed without proof and some of which have been mentioned in the literature. 

First of all, 2-dimensional string small covers are nothing but equivariant connected sum of $\R P^{1} \times \R P^{1}$. 

Secondly, combine the Four Color Theorem and results in \cite{Hu20}, every 3-dimensional simple polytope can be realized as the orbit polytope of a string small cover. In particular, every string small cover over prism is induced by 3-coloring or 4-coloring, and must be of bundle type. 

Thirdly, $P=C_{2}(m_{1}) \times C_{2}(m_{2})$ can be realized as the orbit polytope of a string small cover if and only if $m_{1}m_{2} \equiv 0 \pmod{2}$. 

\begin{example}
The small cover $M_{\R}(C_{2}(4) \times C_{2}(3),\Lambda)$ is string if $\Lambda$ is equivalent to 
\[
\left(
\begin{array}{cc|cc|cc|c}
1 & 0 & 0 & 0 & 1 & 0 & 0 \\
0 & 1 & 0 & 0 & 0 & 1 & 1 \\
\hline
0 & 0 & 1 & 0 & 0 & 0 & 1 \\
0 & 0 & 0 & 1 & 0 & 0 & 1 
\end{array}
\right).
\]
\end{example}

Lastly, each string small cover over product of simplices $\prod_{i=1}^{k} \Delta^{n_{i}}$ is a generalized real Bott manifold, i.e., the total space $B^{n}$ of an $\R P^{n_{i}}$-bundle tower: 
\[
B^{n} \xrightarrow{\R P^{n_{k}}} B^{n-n_{k}} \xrightarrow{\R P^{n_{k-1}}} \cdots \xrightarrow{\R P^{n_{2}}} B^{n_{1}} \xrightarrow{\R P^{n_{1}}} {pt}, 
\]
where $n=\sum_{i=1}^{k} n_{i}$. In this case, calculation based on an explicit formula for $w_{2}$ given in \cite{DU19} (see also \cite{CMO17}) induces the following proposition: 
\begin{proposition}
$P=\prod_{i=1}^{k}\Delta^{n_{i}}$ with $n_{i}\geq 2$ can be realized as the orbit polytope of a string small cover if and only if $n_{i}\equiv 1 \pmod{2}$ for $1 \leq i \leq k$ and $\exists\ i_{0}$ such that $n_{i_{0}}\equiv 3 \pmod{4}$.
\end{proposition}
As shown in examples below, similar characterization does not exist for $P=I^{m} \times \prod_{i=1}^{k} \Delta^{n_{i}}$ with $m \geq 1, n_{i} \geq 2$. 

\begin{example}
The small cover $M_{\R}(I^{1} \times \Delta^{2},\Lambda)$ is string if $\Lambda$ is equivalent to 
\[
\left(
\begin{array}{c|cc|c|c}
1 & 0 & 0 & 1 & 1 \\
\hline
0 & 1 & 0 & 0 & 1 \\
0 & 0 & 1 & 0 & 1 
\end{array}
\right).
\]
\end{example}

\begin{example}
The small cover $M_{\R}(I^{1} \times \Delta^{3} \times \Delta^{4},\Lambda)$ is string if $\Lambda=[\ \mathrm{I}_{8}\ |\ \Lambda_{*}\ ]$ with 
\[
\Lambda_{*}=
\left(
\begin{array}{c|c|c}
1 & 0 & 1 \\
\hline
0 & 1 & 0 \\
0 & 1 & 1 \\
0 & 1 & 1 \\
\hline
0 & 0 & 1 \\
0 & 0 & 1 \\
0 & 0 & 1 \\
0 & 0 & 1 
\end{array}
\right).
\]
\end{example}

\section*{Acknowledgement}
\setcounter{equation}{0}

\noindent The author would like to thank professor Zhi L$\ddot{\mathrm{u}}$ for providing valuable advice.

\mbox{}\\
Qifan Shen\\
School of Mathematical Sciences \\
Fudan University \\
220 Handan Road \\
Shanghai 200433 \\
People's Republic of China \\
E-Mail: qfshen17@fudan.edu.cn \mbox{}\\


\begin{thebibliography}{00}

\bibitem{ABP67} D.W. Anderson; E.H. Brown; F.P. Peterson. \emph{The structure of the Spin cobordism ring}. Ann. of Math. (2) 86 (1967), 271-298.

\bibitem{AH70} M. Atiyah; F. Hirzebruch. \emph{Spin-manifolds and group actions}. Essays on Topology and Related Topics (Memoires dedies Georges de Rham) 18-28. Springer, New York, 1970. 

\bibitem{BB92} G. Blind; R. Blind. \emph{Triangle-free polytopes with few facets}. Arch. Math. (Basel) 58 (1992), no. 6, 599-605.

\bibitem{BP15} V.M. Buchstaber; T.E. Panov. \emph{Toric Topology}. Mathematical Surveys and Monographs, 204. American Mathematical Society, Providence, RI, 2015.

\bibitem{BPR07} V.M. Buchstaber; T.E. Panov; N. Ray. \emph{Spaces of polytopes and cobordism of quasitoric manifolds}. Mosc. Math. J. 7 (2007), no. 2, 219-242, 350.

\bibitem{Bu11} U. Bunke. \emph{String structures and trivialisations of a Pfaffian line bundle}. Comm. Math. Phys. 307 (2011), no. 3, 675-712.

\bibitem{CMO17} S. Choi; M. Masuda; S. Oum. \emph{Classification of real Bott manifolds and acyclic digraphs}. Trans. Amer. Math. Soc. 369 (2017), no. 4, 2987-3011.

\bibitem{CMS10} S. Choi; M. Masuda; D.Y. Suh. \emph{Quasitoric manifolds over a product of simplices}. Osaka. J. Math. 47 (2010), no.1, 109-129.

\bibitem{CMS11} S. Choi; M. Masuda; D.Y. Suh. \emph{Rigidity problems in toric topology: a survey}. Proc. Steklov Inst. Math. 275 (2011), no. 1, 177-190.

\bibitem{CPS10} S. Choi; T. Panov; D. Y. Suh. \emph{Toric cohomological rigidity of simple convex polytopes}. J. Lond. Math. Soc. (2) 82 (2010), no. 2, 343-360.

\bibitem{DJ91} W.M. Davis; T. Januszkiewicz. \emph{Convex polytopes, Coxeter orbifolds and torus actions}. Duke Math. J. 62 (1991), no.2, 417-451.

\bibitem{Do01} N.E. Dobrinskaya. \emph{The classification problem for quasitoric manifolds over a given polytope}. Funct. Anal. Appl. 35 (2001), no. 2, 83-89.

\bibitem{DU19} R. Dsouza; V. Uma. \emph{Results on the topology of generalized real Bott manifolds}. Osaka J. Math. 56 (2019), no. 3, 441-458.

\bibitem{Er14} N.Y. Erokhovets. \emph{Buchstaber invariant theory of simplicial complexes and convex polytopes}, Proc. Steklov Inst. Math. 286 (2014), no. 1, 128-187.

\bibitem{Gr03} B. Grunbaum. \emph{Convex polytopes}. Graduate Texts in Mathematics, 221. Springer-Verlag, New York, 2003.

\bibitem{Ha15} S. Hasui. \emph{On the classification of quasitoric manifolds over dual cyclic polytopes}, Algebr. Geom. Topol. 15 (3) 1387-1437, 2015.

\bibitem{Hi66} F. Hirzebruch. \emph{Topological methods in algebraic geometry}. Grundlehren Math. Wiss. 131, Springer, Berlin 1966.

\bibitem{Ho02} M.J. Hopkins. \emph{Algebraic topology and modular forms}. Proceedings of the International Congress of Mathematicians, Vol. I (Beijing, 2002), 291-317, Higher Ed. Press, Beijing, 2002.

\bibitem{HR95} M.A. Hovey; D.C. Ravenel. \emph{The 7-connected cobordism ring at p=3}. Trans. Amer. Math. Soc. 347 (1995), no. 9, 3473-3502.

\bibitem{Hu20} Z. Huang. \emph{The second Stiefel-Whitney class of small covers}. Chin. Ann. Math. Ser. B 41 (2020), no. 2, 163-176.

\bibitem{Jo02} M. Joswig. \emph{Projectivities in simplicial complexes and colorings of simple polytopes}. Math. Z. 240 (2002), no.2, 243-259.

\bibitem{Ju73} P.E. Jupp. \emph{Classification of certain} 6-\emph{manifolds}. Proc. Cambridge Philos. Soc. 73 (1973), 293-300.

\bibitem{Ki87} T.P. Killingback. \emph{World-sheet anomalies and loop geometry}. Nuclear Phys. B 288 (1987), no. 3-4, 578-588.

\bibitem{KM13} C. Kottke; R. Melrose. \emph{Equivalence of string and fusion loop-spin structures}.
arXiv:1309.0210 (2013).

\bibitem{LLP18} I. Y. Limonchenko; Z. L\"{u}; T. E. Panov. \emph{Calabi-Yau hypersurfaces and SU-bordism}. (Russian. Russian summary)
English version published in Proc. Steklov Inst. Math. 302 (2018), no. 1, 270-278. Tr. Mat. Inst. Steklova 302 (2018), Topologiya i Fizika, 287-295.

\bibitem{Lu17} Y. Lu. \emph{On cobordism of generalized real Bott manifolds}. arXiv:1710.00562 (2017).

\bibitem{LP16} Z. L\"{u}; T. Panov. \emph{On toric generators in the unitary and special unitary bordism rings}. Algebr. Geom. Topol. 16 (2016), no. 5, 2865-2893.

\bibitem{LY11} Z. L\"{u}; L. Yu. \emph{Topological types of 3-dimensional small covers}. Forum Math. 23 (2011), no. 2, 245-284.

\bibitem{MG95} M. Mahowald; V. Gorbounov. \emph{Some homotopy of the cobordism spectrum MO$\langle8\rangle$}. Homotopy theory and its applications (Cocoyoc, 1993), 105-119, Contemp. Math., 188, Amer. Math. Soc., Providence, RI, 1995.

\bibitem{OR70} P. Orlik; F. Raymond. \emph{Actions of the torus on 4-manifolds. I}. Trans. Amer. Math. Soc., 152 (1970), 531-559.

\bibitem{ST05} S. Stolz; P. Teichner. \emph{The spinor bundle on loop space}. MPIM preprint (2005).

\bibitem{Wa15} K. Waldorf. \emph{String geometry vs. spin geometry on loop spaces}. J. Geom. Phys. 97 (2015), 190-226.

\bibitem{Wa16} K. Waldorf. \emph{Spin structures on loop spaces that characterize string manifolds}. Algebr. Geom. Topol. 16 (2016), no. 2, 675-709.

\bibitem{Wall66} C.T.C. Wall. \emph{Classification problems in differential topology. V. On certain} 6-\emph{manifolds}. Invent. Math. 1 (1966), 355-374.

\bibitem{Wi88} E. Witten. \emph{The index of the Dirac operator in loop space}. Elliptic curves and modular forms in algebraic topology (Princeton, NJ, 1986), 161-181, Lecture Notes in Math., 1326, Springer, Berlin, 1988.

\bibitem{Zi98} G.M. Ziegler. \emph{Lectures on polytopes}. Graduate Texts in Mathematics, 152. Springer-Verlag, New York, 1995.

\end{thebibliography}
\end{document}